\documentclass[11pt,a4paper,reqno]{article}
\usepackage{amsmath}
\usepackage{amsthm}
\usepackage{enumerate}
\usepackage{amssymb}

\usepackage{verbatim}
\usepackage[latin1]{inputenc}

\usepackage{graphicx,color}

\def\be{\begin{equation}}
\def\ee{\end{equation}}
\def\bea{\begin{equation*}}
\def\eea{\end{equation*}}
\def\begs{\begin{split}}
\def\ends{\end{split}}

\newtheorem{thm}{Theorem}
\newtheorem*{thm*}{Theorem}
\newtheorem{lma}[thm]{Lemma}
\newtheorem{cor}[thm]{Corollary}
\newtheorem{prop}[thm]{Proposition}

\newtheorem{df}{Definition}

\theoremstyle{remark}
\newtheorem{preremark}[thm]{Remark}
\newtheorem{preex}[thm]{Example}

\newenvironment{remark}{\begin{preremark}}{\end{preremark}}

\title{On the Chemical Distance in Critical Percolation II}
\author{Michael Damron \thanks{The research of M. D. is supported by NSF grant DMS-0901534.} \\ \small{Georgia Tech} \\ \small{Indiana University, Bloomington}  \and Jack Hanson\thanks{The research of J.S. is supported by an AMS-Simons Travel grant.} \\ \small{Georgia Tech} \\ \small{City College, NY} \and Philippe Sosoe\thanks{The research of P. S. is supported by the Center for Mathematical Sciences and Applications at Harvard University.} \\ \small{CMSA, Harvard}}

\begin{document}
\maketitle

\abstract{We continue our study of the chemical (graph) distance inside large critical percolation clusters in dimension two. We prove new estimates, which involve the three-arm probability, for the point-to-surface and point-to-point distances. We show that the point-to-point distance in $\mathbb{Z}^2$ between two points in the same critical percolation cluster has infinite second moment. We also give quantitative versions of our previous results comparing the length of the shortest crossing to that of the lowest crossing of a box.}

\section{Introduction}
In this paper, we continue our study of the graph distance inside large critical percolation clusters in $\mathbb{Z}^2$, initiated in \cite{DHS15}. The graph distance between two subsets $A$, $B$ in $\mathbb{Z}^2$ is defined as the minimal number of edges in any path remaining in one percolation cluster, connecting $A$ to $B$. This quantity is commonly referred to as the ``chemical distance'', to distinguish it from the Euclidean distance in $\mathbb{Z}^2$.

In \cite{DHS15}, we considered the distance between the sides of the square $B_n(0)=[-n,n]^2$, conditioned on the existence of a crossing, obtaining an upper bound on the length of the shortest crossing(s) from the left to the right side. This problem had been previously considered by H. Kesten and Y. Zhang \cite{KZ} (see also \cite[Problem 3.3]{schramm-ICM}). The result in \cite{DHS15} is that the length $S_n$ of the shortest crossing(s) is asymptotically ``infinitely smaller'' than the length $L_n$ of the lowest crossing of $B_n(0)$:
\begin{thm*}[\cite{DHS15}]
Conditioned on the event 
\begin{equation}\label{eqn: Hn}
H_n=\{\text{there exists a horizontal open crossing on } [-n,n]^2\},
\end{equation}
one has 
\begin{equation}\label{eqn: prob}
\frac{S_n}{L_n} \rightarrow 0 \text{ in probability.}
\end{equation}
\end{thm*}
The volume $L_n$ is known to be of order $Cn^2\pi_3(n)$ with high probability as $n\rightarrow \infty$ \cite{KZ,MZ} (see also \cite[Section 7]{DHS15}), where $\pi_3(n)$ is the probability that the origin $(0,0)$ is connected to $\partial B_n(0) =\{x\in\mathbb{Z}^2:|x|_\infty = n\}$ by two disjoint open paths and a closed dual path. (On the triangular lattice, the analogous probability is $\pi_3(n)=n^{-2/3+o(1)}$.) Thus, the chemical distance between vertical sides of $B_n(0)$ is $o(n^2\pi_3(n))$, with high probability. A similar statement holds in expectation, as was also shown in \cite{DHS15}.

In many applications, in particular the study of random walks, it is more natural to consider the \emph{radial} chemical distance; that is, the distance from the origin to the exterior of the $\ell_\infty$ ball $B_n(0)$. To the best of our knowledge, there are no bounds for the radial distance in the literature. It is not difficult to estimate this distance in terms of the two-arm exponent. Our first result improves such a bound to one in terms of the (smaller) three-arm exponent. Unlike the case of the shortest crossing, there is now no ``lowest'' crossing of $B_n(0)$ to compare to. 
\begin{thm}\label{thm: radial}
Let $A_n$ be the event that there is an open connection from the origin $(0,0)$ to $\partial
B_n(0)$. On $A_n$, the random variable $S_{B(0,n)}$ is defined as the chemical distance between the origin and $\partial B_n(0)$.

There is a constant $C>0$ independent of $n$ such that
\begin{equation}\label{eqn: radial}
\mathbf{E}[S_{B_n(0)}\mid A_n]\le Cn^2\pi_3(n).
\end{equation}
\end{thm}
We expect that it is possible to improve the estimate \eqref{eqn: radial}, replacing the right side by a quantity that is $o(n^2\pi_3(n))$, using the techniques developed in \cite{DHS15} and further refined in this paper (see Theorem \ref{thm: quantitative} below). We leave such an improvement for later investigation.

Using the approach developed in the proof of Theorem \ref{thm: radial}, we also obtain an estimate on the unrestricted point-to-point distance between two points $x,y\in \mathbb{Z}^2$, conditioned on the event $\{x\leftrightarrow y\}$ that $x$ and $y$ are connected by an open path. The adjective ``unrestricted'' refers to the connections being allowed to be anywhere on the plane, not just inside some large box.
\begin{cor}\label{thm: pt2pt}
Let $x, y\in \mathbb{Z}^2$, and $d = |x-y|_1$ be the $\ell_1$ distance between them. On $\{x\leftrightarrow y\}$, denote by $\mathrm{dist}_{chem}(x,y)$ the (random) chemical distance between $x$ and $y$. Then there are constants $C, c_1>0$ such that for any $\lambda>0$ and $x,y$,
\begin{equation}\label{eqn: chem}
\mathbf{P}(\mathrm{dist}_{chem}(x,y)\ge \lambda d^2\pi_3(d) \mid x \leftrightarrow y)\le C\lambda^{-c_1}
\end{equation}
\end{cor}

The constant $c_1>0$ in \eqref{eqn: chem} we obtain in our proof is quite small. The next proposition shows, however, that this estimate on the tail probability cannot be much improved. Indeed, the distribution of the unrestricted chemical distance is not concentrated. The second moment is infinite, regardless of the Euclidean distance between $x$ and $y$:
\begin{prop}\label{prop: nomoment}
With the notation of the previous corollary, we have
\begin{equation}\label{eqn: nomomentE}
\mathbf{E}[(\mathrm{dist}_{chem}((0,0),(1,0)))^2\mid (0,0)\leftrightarrow (1,0)]=\infty.
\end{equation}
\end{prop}
A similar result holds for the chemical distance between any two fixed vertices. As will be clear from the proof, the second moment has no special significance in \eqref{eqn: nomomentE}. The fractional moment $\mathbf{E}[\mathrm{dist}_{chem}((0,0),(1,0))^{2-\delta}\mid (0,0)\leftrightarrow (1,0)]$ can be shown to be infinite also for small $\delta>0$. On the other hand, we do not know whether the expected distance ($\delta =1$) is finite.

In the last part of this paper, we return to the problem of the chemical distance across the box $B(0,n)$, and improve our previous result \cite{DHS15} on the expected length of the shortest crossing to a quantitative result:
\begin{thm}\label{thm: quantitative}
Let $H_n$ be as in \eqref{eqn: Hn}. For any $0\le c_2<1/4$, there is a constant $C>0$ such that
\begin{equation}\label{eqn: ELn}
\mathbf{E}[ S_n \mid H_n] \le \frac{C}{(\log n)^{c_2}} \cdot n^2\pi_3(n) \text{ for all }n.
\end{equation}
\end{thm}
We can follow the approach in \cite[Section 7]{DHS15} to obtain the following corollary, improving on \eqref{eqn: prob}.
\begin{cor}\label{cor: quantitative} For $0<c_3<1/4$, we have
\[\mathbf{P}\left( S_n \le \frac{1}{(\log n)^{c_3}}\cdot L_n  \mid H_n \right)\rightarrow 1 \text{ as } n \to \infty.\]
\end{cor}
As in the derivation of  \cite[Corollary 3]{DHS15}, to pass from Theorem \ref{thm: quantitative} to the corollary, we must show that  $L_n$ is of order at least $n^2\pi_3(n)$ with high probability. This follows readily from the estimate for the lower tail of $L_n$ obtained in \cite[Lemma 24]{DHS15}.

\begin{remark}
The lowest crossing is distinguished in the sense that there is a simple characterization of its vertices in terms of arm events. Even though it is believable that there are other self-avoiding open paths inside crossing clusters whose typical lengths are different (say larger) than that of the lowest crossing, it is not easy to show this. However, the proof of the above results can be modified to show that the longest open crossing has length at least $(\log n)^{c_4}$ times that of the lowest with high probability.
\end{remark}

\subsection{Overview of the paper}
After setting some notation at the end of this introduction, we prove Theorem \ref{thm: radial} in Section \ref{sec: radial}. The main difficulty, compared to the case of the shortest crossing of a rectangle, is that there is no natural lowest path appearing in every configuration. We replace the lowest path by a deterministic construction based on successive innermost open circuits around the origin. Such circuits consist of three-arm points: points with two disjoint open connections, and one closed dual connection to a large distance. On the event $A_n(0)$, the circuits themselves can be connected to each other by paths made of three-arm points. In a typical configuration, $B_n(0)$ contains on the order of $\log n$ disjoint circuits around the origin. To avoid an additional logarithmic factor on the right side in \eqref{eqn: radial}, we keep track of the number of arms generated on large scales in a configurations with many circuits.

The proof of Corollary \ref{thm: pt2pt} in Section \ref{sec: chem} follows the strategy for Theorem \ref{thm: radial}. We establish an estimate for the point-to-point distance between two points $x$ and $y$ in some fixed box of size $n$ (see Proposition \ref{prop: superman}). The estimate \eqref{eqn: chem} quickly follows from this, by noting that the connection between $x$ and $y$ can be forced to lie inside a box of size $2^K|x-y|_1$ by placing a dual circuit around this box.

The proof of Proposition \ref{prop: nomoment} turns on the observation that the event ``$(0,0)$ and $(1,0)$ are connected in $B_{2^k}(0)$ but are not connected in $B_{2^{k-1}}(0)$'' has probability comparable to $\pi_4(2^{k-1})$, the 4-arm probability to distance $2^{k-1}$. This last probability can be compared to the five-arm probability, whose exact asymptotic behavior is known, allowing us to derive \eqref{moment}.

In Section \ref{sec: quantitative}, we prove Theorem \ref{thm: quantitative} using ideas introduced in \cite{DHS15}. The innovation is a different construction of an event $E_k$ implying the existence of a shortcut around a portion of the lowest crossing resulting in a gain of a factor $\epsilon$. See Section \ref{sec: Ek} for the definition of the event and explanatory figures. The advantage of the new construction is that the $\epsilon$-dependence of the probability of $E_k$ is better than for the corresponding event in \cite{DHS15}.

\subsection{Notation}
On the square lattice $(\mathbb{Z}^2,\mathcal{E}^2)$, let $\mathbf{P}$ be the critical bond percolation measure $\prod_{e \in \mathcal{E}^2} \frac{1}{2}(\delta_0 + \delta_1)$ on $\Omega = \{0,1\}^{\mathcal{E}^2}$. An edge $e$ is said to be open in the configuration $\omega$ if $\omega(e)=1$; it is closed otherwise. A (lattice) path is a sequence $\{v_0,e_1,v_1, \ldots, v_{N-1},e_N,v_N\}$ such that for all $k=1, \ldots, N$, $|v_{k-1}-v_k|_1=1$ and $e_k = \{v_{k-1},v_k\}$. A circuit is a path with $v_0=V_N$. For such paths we denote $\# \gamma = N+1$ (the number of vertices) and $\#E(\gamma)$ the number of edges in $\gamma$; that is, the number of elements in the edge set $E(\gamma)$ of $\gamma$. If $V \subset \mathbb{Z}^2$ then we say that $\gamma \in V$ if $v_k\in V$ for $k=0, \ldots, N$. Last, a path $\gamma$ is said to be (vertex) self-avoiding if $v_i=v_j$ implies $i=j$ and a circuit is (vertex) self-avoiding if $v_i=v_j \neq v_0$ implies $i=j$. Given $\omega \in \Omega$, we say that $\gamma$ is open in $\omega$ if $\omega(e_k)=1$ for $k=1, \ldots, N$.

The dual lattice is written $((\mathbb{Z}^2)^*,(\mathcal{E}^2)^*)$, where $(\mathbb{Z}^2)^* = \mathbb{Z}^2 + (1/2)(\mathbf{e}_1+\mathbf{e}_2)$ with its nearest-neighbor edges. Here, we have denoted by $\mathbf{e}_i$ the coordinate vectors:
\[ \mathbf{e}_1 = (1,0), \ \mathbf{e}_2 = (0,1).\]
Given $\omega \in \Omega$, we obtain $\omega^* \in \Omega^* = \{0,1\}^{(\mathcal{E}^2)^*}$ by the relation $\omega^*(e^*) = \omega(e)$, where $e^*$ is the dual edge that shares a midpoint with $e$. For any $V \subset \mathbb{Z}^2$ we write $V^* \subset (\mathbb{Z}^2)^*$ for $V + (1/2)(\mathbf{e}_1+\mathbf{e}_2)$. 

For any two subsets $A$, $B$ of $\mathbb{Z}^2$, we denote by $\{A\leftrightarrow B\}$ the event that there exists an open path of with one endpoint in $A$, and the other in $B$. If $S\subset \mathbb{Z}^2$ is a third subset of $\mathbb{Z}^2$, we denote by $\{A \leftrightarrow_{S}  B\}$ the event that $A$ and $B$ are connected by an open path whose vertices are all in the set $S$. If $A$ or $B$ is a single point, for example if $A=\{x\}$, we extend the notation above by writing $\{x\leftrightarrow B\}$ for $\{\{x\}\leftrightarrow B\}$. On $\{A\leftrightarrow B\}$, the random variable $\mathrm{dist}_{chem}(A,B)$, the chemical distance, is the least number of edges in any open path connecting $A$ to $B$. On the event $\{A \leftrightarrow_S B\}$, $\mathrm{dist}^S_{chem}(A,B)$ is the least number of edges in any open path connecting $A$ to $B$, and with all of its vertices in $S$.

For $x=(x_1,x_2)\in \mathbb{Z}^2$, we define
\[B_n(x) = \{y\in\mathbb{Z}^2: |y-x|_\infty \le n\}.\]
When $x$ is the origin $(0,0)$, we sometimes abbreviate $B_n((0,0))$ by $B_n(0)$ or $B_n$. We denote by $\partial B_n(x)$ the set
\[\partial B_n(x) =\{y\in\mathbb{Z}^2: |y-x|_\infty = n\}.\]

We end by defining some ``arm events'' inside $B_n(x)$. $A_n(x)$ is the event that $x$ is connected to $\partial B_n(x)$  by an open path. The probability of this event in critical percolation is denoted by $\pi_1(n)$:
\begin{equation}\label{eqn: pi1}
\pi_1(n) =\mathbf{P}(A_n(0)).
\end{equation}
The notation $\pi_3(n)$ is used to denote the probability of the event that the edge $\{(0,0),(1,0)\}$ is connected to $\partial B_n(0)$ by two disjoint open paths, and the dual edge $\{(\frac{1}{2},-\frac{1}{2}),(\frac{1}{2},\frac{1}{2})\}$ is connected by a closed dual path to $\partial B_n(0)^*$. Similarly, $\pi_4(n)$ is the probability that $\{(0,0),(1,0)\}$ is connected to $\partial B_n(0)$ by two disjoint open paths, and $\{(\frac{1}{2},-\frac{1}{2}),(\frac{1}{2},\frac{1}{2})\}$ is connected to $\partial B_n(0)^*$ by two disjoint closed dual paths. The arms appear in alternating order.

In this paper, constants (like $C$, $C'$, and so on) may differ from line to line, but never depend on any parameters involved in the argument in which they appear. If $(a_n)$ and $(b_n)$ are sequences of numbers, then the notation $a_n \asymp b_n$ will denote that there are positive constants $C,C'$ such that $Ca_n \leq b_n \leq C'a_n$ for all $n$.

\section{The radial chemical distance}
\label{sec: radial}
In this section, we prove Theorem \ref{thm: radial}, giving  an estimate for the ``radial'' chemical distance. Recall that $A_n = A_n(0)$ is the event that the origin is connected to $\partial B_n(0)$ by an open path. On $A_n$, define $S_{B_n(0)}$ as the least number edges in any open path from $0$ to $\partial B_n(0)$.

\subsection{Open circuits around $0$}\label{sec: c}
On $A_n$, let $C_0$ be the event that there is an open circuit around $0$ in $B_n=B_n(0)$. Note that on $A_n\cap C_0^c$, duality implies that there is, in addition to an open path, a closed dual path $\mathfrak{c}$ from a dual neighbor of the origin to $\partial B_n^*$. 

On $C_0$, let $K$ be the maximum number of disjoint open circuits around $0$ in $B_n$. For $1\le k \le K$, let $\mathcal{C}_k$ be the $k$-th innermost open circuit around 0. ($\mathcal{C}_1$ is the innermost open circuit around $0$ in $B_n$, $\mathcal{C}_2$ is the innermost open circuit around $0$ in $\mathrm{ext}(\mathcal{C}_1)\cap B_n$, and so on.) 

By minimality, each dual edge crossing $\mathcal{C}_1$ is connected to a vertex of the dual lattice adjacent to the origin by a closed dual path. For $k=1,\ldots, K-1$, each dual edge crossing $\mathcal{C}_{k+1}$ is connected to a dual edge crossing $\mathcal{C}_k$ by a closed dual path in $\mathrm{ext}\, \mathcal{C}_k$. Finally, $(\partial B_n)^*$ is connected to a dual edge crossing $\mathcal{C}_K$ by a closed dual path.

\subsection{Estimate for $S_{B_n(0)}$: case where $C_0^c$ holds}\label{sec: C0c}
On $A_n\cap C_0^c$, there is an open path $\tilde{\sigma}_n$ from $0$ to $\partial B_n$ which is closest to the ``left side'' of the closed path $\mathfrak{c}$ (from Section \ref{sec: c}) from the origin to $\partial B_n^*$. Let $R$ be the region strictly between $c$ and $\tilde{\sigma}_n$ (in the counter-clockwise ordering of paths starting from $\mathfrak{c}$). By definition, there is no open path from $0$ to $\partial B_n(0)$ in $R\cup \mathfrak{c}$ containing any vertex of $R$. Therefore by duality, for each edge $e$ of $\tilde{\sigma}_n$, there is a closed dual path in $R$ from $e^*$ to $\mathfrak{c}$. From this we have

\begin{lma}\label{lem: pizza_head}
Suppose $A_n\cap C_0^c$ holds. Then each edge $e\in \tilde{\sigma}_n$ has three disjoint arms to distance 
\[k=\min\{\mathrm{dist}(e,0),\mathrm{dist}(e,\partial B_n)\},\] two open arms and one closed dual arm.
\begin{proof}
The existence of the two open arms follows from the fact that $e\in \tilde{\sigma}_n$. The closed arm is obtained by following the above-mentioned dual closed path from $e^*$ to $\mathfrak{c}$, and then following $\mathfrak{c}$ toward $0$ or $\partial B_n$.
\end{proof}
\end{lma}

Let $E_3(e,k)$ be the event in the lemma: that $e$ has three such disjoint arms to distance $\min\{\text{dist}(e,0),\text{dist}(e,\partial B_n)\}$. Then, by the above remarks,
\begin{align}
\mathbf{E}[S_{B_n(0)}; C_0^c \mid A_n] &\le \mathbf{E}[\# \tilde{\sigma}_n; C_0^c\mid A_n] \nonumber \\
&\le C \sum_{k=1}^{\lfloor n/2\rfloor}\sum_{e:d(e,0)\wedge d(e,\partial B_n)=k}\mathbf{P}(E_3(e,k)\mid A_n). \label{eqn: e3k}
\end{align}
Here, we write $d$ for $\mathrm{dist}$. Furthermore, by independence,
\begin{equation}\label{eq: nachos_bellegrande}
\mathbf{P}(E_3(e,k)\mid A_n) \le \frac{\mathbf{P}(A_{\mathrm{dist}(0,e)/2})\mathbf{P}(E_3(e,k/2))\mathbf{P}(A_{\mathrm{dist}(0,e)/2+k,n})}{\mathbf{P}(A_n)}
\end{equation}
Here, if $k\le l$, then $A_{k,l}$ denotes the event that there is an open path from $\partial B_k(0)$ to $\partial B_l(0)$. If $k>l$, then $A_{k,l}$ is defined to be the sure event. Since $k\le \mathrm{dist}(0,e)$, we have by quasimultiplicativity \cite[Proposition~12.2]{nolin}
\[\mathbf{P}(A_{\mathrm{dist}(0,e)/2})\mathbf{P}(A_{\mathrm{dist}(0,e)/2+k,n}) \le C\mathbf{P}(A_n).\]

Placing this inequality in \eqref{eq: nachos_bellegrande} and returning to \eqref{eqn: e3k}, we obtain a bound of
\begin{equation}\label{eqn: artem}
\mathbf{E}[S_{B_n(0)}; C_0^c ] \leq C\sum_{k=1}^{\lfloor n/2\rfloor}k\pi_3(k)  \le Cn^2\pi_3(n).
\end{equation}
What we have described here is a standard gluing argument to remove the conditioning on $A_n$. We will use similar arguments in Section~\ref{sec: reckoning}, but provide the details in the above case only.

To derive \eqref{eqn: artem}, we have used the following lemma:
\begin{lma}\label{lem: artem}
There exists $C>0$ such that for all $L\in \mathbb{Z}_+$,
\[\sum_{l=1}^L l\pi_3(l) \le CL^2\pi_3(L).\]
\begin{proof}
By Proposition~\ref{prop: boss} below, we can choose $\beta < 2$ and $c>0$ such that $\pi_3(l,L)\ge c(l/L)^\beta$, and use quasimultiplicativity:
\[\sum_{l=1}^L l \pi_3(l) \asymp \pi_3(L)\sum_{l=1}^L \frac{l}{\pi_3(l,L)}\le C L^\beta \pi_3(L)\sum_{l=1}^L l^{1-\beta} \le CL^2\pi_3(L).\]
\end{proof}
\end{lma}

\subsection{Estimate on $C_0$}
On $C_0$, we use a more intricate construction to obtain a path $\tilde{\sigma}_n$. 

There is an open path from 0 to $\mathcal{C}_1$ inside $\mathrm{int}(\mathcal{C}_1)$. On the other hand, by duality, every dual edge $e^*$ crossing $\mathcal{C}_1$ has a dual closed connection inside $\mathrm{int}(\mathcal{C}_1)$ from $e^*$ to a dual neighbor of $0$. We let $\mathfrak{c}_1$ be such a closed path from the origin in $\mathrm{int}(\mathcal{C}_1)$ that is closest to the right side of the open path, and let  $\tilde{\sigma}_n^1$ be the open path from $0$ to $\mathcal{C}_1$ in $\mathrm{int}(\mathcal{C}_1)$ that is closest to the left side of the closed path $\mathfrak{c}_1$. For each edge $e\in \tilde{\sigma}^1_n$, there is a closed dual path from $e^*$ to $\mathfrak{c}_1$ inside the region $R_1$ between $\mathfrak{c}_1$ and $\tilde{\sigma}_n^1$, in counterclockwise order. From this, we obtain:
\begin{lma}\label{lma: firstct}
If $e\in \tilde{\sigma}_n^1$, then $e$ has two disjoint open arms and one closed dual arm to distance $\mathrm{dist}(0,e)$.
\begin{proof}
One open arm is obtained by following the open path $\tilde{\sigma}^1_n$ from $e$ to the origin. The second open arm is obtained by following $\tilde{\sigma}^1_n$ in the opposite direction until the open circuit $\mathcal{C}_1$, and then going around $\mathcal{C}_1$ in either direction. Since $\mathcal{C}_1$ goes around $0$, this results in an open arm that leaves the box $B_{\mathrm{dist}(0,e)}(e)$. The closed arm is obtained by following the closed dual path from $e^*$ to $\mathfrak{c}_{1}$ (mentioned in the preceding paragraph) until $\mathfrak{c}_{1}$, and then $\mathfrak{c}_{1}$ to the origin.
\end{proof}
\end{lma}

Similarly, if $K>1$, there is an open path in $B_n$ from $\mathcal{C}_K$ to $\partial B_n$. By definition of $K$, there is no open circuit around 0 in $B_n \cap \mathrm{ext}(\mathcal{C}_K)$, and so by duality, there is a closed dual path $\mathfrak{c}_{K+1}$ from $(\partial B_n)^*$ to some dual edge $e^*$ with $e\in \mathcal{C}_K$. Define $\tilde{\sigma}^{K+1}_n$ to be the open path from $\mathcal{C}_K$ to $\partial B_n$ which is closest to the left side of $\mathfrak{c}_{K+1}$. For each $e\in \tilde{\sigma}^{K+1}_n$, there is a closed dual path from $e^*$ to $\mathfrak{c}_{K+1}$.

\begin{lma}\label{lma: lastct}
Let $e\in \tilde{\sigma}^{K+1}_n$ and $M=\min(\mathrm{dist}(0,e),\mathrm{dist}(e,\partial B_n))$. The edge $e$ has two disjoint open arms and one closed dual arm to distance $M$.
\begin{proof}
By following $\tilde{\sigma}_n^{K+1}$ in opposite directions from $e$, we obtain two open arms, one to $\partial B_n(0)$, and one to the circuit $\mathcal{C}_K$. Once $\mathcal{C}_K$ is reached, this arm can be extended along the circuit to $\partial B_M(e)$. Note that $0\notin B_M(e)$, so the circuit $\mathcal{C}_K$ reaches outside $B_M$. The dual closed arm is found as before: by following a closed dual path from $e^*$ to $\mathfrak{c}_{K+1}$, and then this latter path to $\partial B_n(0)^*$.
\end{proof}
\end{lma}

Until Section~\ref{sec: reckoning} then, we assume $K>1$. 
For $k=1,\ldots,K-1$, there is an open path $\gamma_{k+1}$ in $\mathrm{ext}(\mathcal{C}_k)\cap \mathrm{int}(\mathcal{C}_{k+1})$ joining the two circuits. Since $\mathcal{C}_{k+1}$ is the innermost open circuit around the origin in $\mathrm{ext}(\mathcal{C}_k)$, there is a closed dual path connecting these circuits inside $\mathrm{ext}(\mathcal{C}_k)\cap \mathrm{int}(\mathcal{C}_{k+1})$. We let $\mathfrak{c}_{k+1}$ be the closed path that is closest to the right side of $\gamma_{k+1}$, and $\tilde{\sigma}_n^{k+1}$ be the open path that is closest to the left side of $\mathfrak{c}_{k+1}$. For every edge $e\in \tilde{\sigma}^{k+1}_n$, by duality, there is a closed dual path from $e^*$ to $\mathfrak{c}_{k+1}$ inside the region $R_{k+1}$ between these two paths (in the counterclockwise order).

\begin{figure}[h]
\centering
\includegraphics[scale = 0.75]{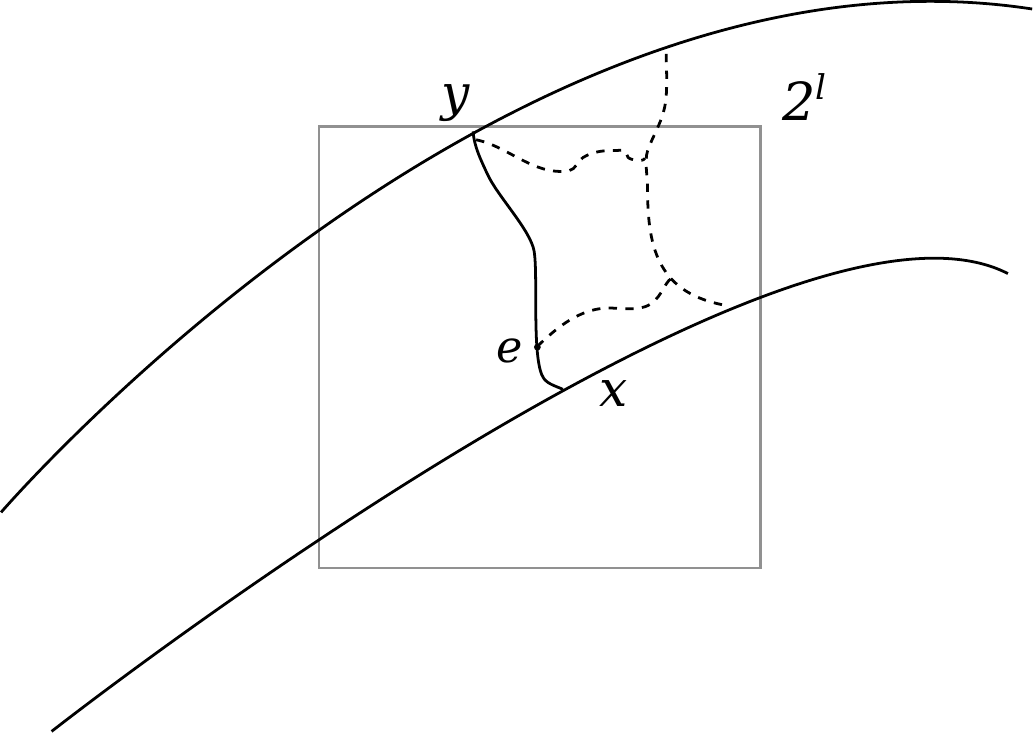}
\label{circuits}
\caption{Depiction of the argument in the proof of Lemma~\ref{lma: middlect}. The box is $B_{2^l}(e)$, where $l = l''$ ($e$ is closer to both the points $x$ and $y$ than it is to 0 or $\partial B_n$). $e$ has two disjoint open arms to distance $2^l$ and one closed dual arm formed by following a closed path to the vertical closed path that connects the circuits, and then a closed dual path to the point $y$. The box $B_{2^l}(e)$ has four disjoint open arms to distance $M$.}
\label{fig: fig_1}
\end{figure}

\begin{lma}\label{lma: middlect}
Let $M=\min(\mathrm{dist}(0,e),\mathrm{dist}(e,\partial B_n))$. Suppose $e\in \tilde{\sigma}_n^k$ for $1<k\leq K$. There is  $1\le l\le \lfloor  \log M \rfloor$ such that
\begin{enumerate}
\item There are two disjoint open arms and one closed dual arm from $e$ to $B_{2^{l-1}}(e)$.
\item If $l < \lfloor \log M \rfloor$, there are four disjoint open arms from $\partial B_{2^l}(e)$ to $\partial B_{M}(e)$.
\end{enumerate}
\begin{proof}
See Figure~\ref{fig: fig_1} for an illustration of the following argument. Let $x$ be the endpoint of $\tilde{\sigma}_n^{k}$ with the least Euclidean distance to either endpoint of $e$ (breaking ties arbitrarily), and let $y$ be the other endpoint. Let $l''$ be the smallest $l'$ such that $2^{l'}\ge M$ or $B_{2^{l'}}(e)$ contains both $x$ and $y$. 
Suppose first that $2^{l''}\ge M$. In this case, we let $l= \lfloor \log M\rfloor$. The two ends of $\tilde{\sigma}_n^{k}$ inside $B_{2^{l-1}}(e)$ form two open arms from $e$ to $\partial B_{2^l}(e)$. The edge of $\tilde{\sigma}_n^{k}$ with endpoint $y$ has a dual edge connected by a dual closed path to $\mathfrak{c}_k$. Since the same is true of the edge $e^*$, we obtain a closed arm from $e^*$ which extends at least to distance $\mathrm{dist}(e,y)\ge 2^{l-1}$.

If $2^{l''}<M$, we let $l=l''$. As in the previous case, we obtain two open arms and a dual closed arm from $e$ to distance $2^{l-1}$. In addition, both $x$ and $y$ are contained in $B_{2^{l}}(e)$, and each lies on one of the open circuits $\mathcal{C}_k$ and $\mathcal{C}_{k+1}$ around 0, whereas $0\notin B_{2^{l-1}}(e)$ by the condition $2^l< M$. By following the circuits in both directions starting from $x,y$, we obtain four open arms from $\partial B_{2^l}(e)$ to $\partial B_{M}(e)$.
\end{proof}
\end{lma}

Before we move on, we need the following final lemma to deal with the simpler case of the edges inside one of the $\mathcal{C}_k$'s:
\begin{lma}\label{lma: circuits}
Suppose $C_0$ occurs and $e\in \mathcal{C}_k$ for some $1\le k\le K$. As before, let $M=\min(\mathrm{dist}(0,e),\mathrm{dist}(e,\partial B_n))$. There is $1\le l\le \lceil \log M\rceil$ such that:
\begin{enumerate}
\item $e$ has two disjoint open arms and one closed dual arm to $\partial B_{2^{l-1}}(e)$
\item If $2^l < M$, there are four disjoint open arms from $\partial B_{2^l}(e)$ to $\partial B_M(e)$.
\end{enumerate}
\begin{proof}
By duality, since $e\in \mathcal{C}_k$, there is a closed dual path in $\mathrm{int}(\mathcal{C}_k)$ connecting $e^*$ to a dual neighbor of $0$ if $k=1$ or to the dual of some edge $e'\in \mathcal{C}_{k-1}$ if $k>1$. Let $l$ be the minimum $l'$ such that $2^{l'}\ge \mathrm{dist}(e,e')$ if $k>1$ (or $2^l\ge \mathrm{dist}(e,0)$ if $k=1$). Then there are three arms from $e$ to $\partial B_{2^l}(e)$.

If $2^l < M$, since $e'\in B_{2^l}(e)$, we can find four arms from $\partial B_{2^l}(e)$ by following the circuits $\mathcal{C}_k$ and $\mathcal{C}_{k-1}$ from $e$ and $e'$ in both directions.
\end{proof}
\end{lma}

\subsection{Reckoning}\label{sec: reckoning}
In this section we do a final calculation. In \eqref{eqn: artem}, we have already dealt with the case when $C_0^c$ holds. By the constructions in the previous section, on $C_0^c$, we can define an open path $\tilde{\sigma}_n$ from $0$ to $\partial B_n$ by concatenating $\tilde{\sigma}_n^1$,\ldots $\tilde{\sigma}_n^{K+1}$ and pieces of $\mathcal{C}_1,\ldots \mathcal{C}_K$ linking them into a self-avoiding path. It remains to estimate
\begin{align}
\mathbf{E}[ S_{B_n(0)}; C_0\mid A_n ]&\le \mathbf{E}[\#\tilde{\sigma}_n; C_0\mid A_n] \nonumber \\
&\le \sum_{e\in B_n}\mathbf{P}(e\in (\cup_{k=1}^K \mathcal{C}_k) \cup (\cup_{k=1}^{K+1}\tilde{\sigma}_n^k)\mid A_n).\label{eqn: labor-union}
\end{align}

We will apply the lemmas in the previous section to bound the probability in the sum by the alternating 3-arm probability to distance $M=M(e)=\min(\mathrm{dist}(0,e),\mathrm{dist}(e,\partial B_n))$.

Let $\mathcal{A}_1(e)$ be the event $\{e\in \tilde{\sigma}_n^1\}$. Applying Lemma \ref{lma: firstct} and standard gluing constructions to remove the conditioning (as in the argument below Lemma~\ref{lem: pizza_head}), we have
\[\mathbf{P}(\mathcal{A}_1(e) \mid A_n)\le \mathbf{P}(A_3(e,M(e)) \mid A_n) \asymp \pi_3(\mathrm{dist}(0,e))\le \pi_3(M(e)).\]

Let $\mathcal{A}_2(e)$ be the event $\{e\in \tilde{\sigma}_n^{K+1}\}$. Using Lemma \ref{lma: lastct} and gluing constructions, we obtain
\[\mathbf{P}(\mathcal{A}_2(e) \mid A_n) \le \pi_3(M(e)).\]

We turn to the event $A_3(e)=\{e\in \cup_{k=2}^K \tilde{\sigma}_n^k\}$. Using Lemma \ref{lma: middlect}, we have, summing over choices of $1\le l\le \lfloor\log M\rfloor$ and using independence and gluing:
\[\mathbf{P}(A_3(e) \mid A_n)\le C \sum_{l=1}^{\lfloor \log M\rfloor} \pi_3(2^{l-1})\pi_4'(2^l,M),
\]
where $\pi'_4$ denotes the monochromatic 4-arm probability. Using Reimer's work as in \cite[p. 1291]{nolinbeffara} (directly above Theorem~5 there, where the authors derive $\alpha_j' \geq \alpha_j$), we have $\pi_4'\le \pi_4$, where $\pi_4$ is the alternating 4-arm probability. So, using Reimer's inequality \cite{reimer} and quasimultiplicativity, for some $\epsilon>0$, the above sum is bounded by up to constant factor by
\begin{align*}
\sum_{l=1}^{\lfloor \log M\rfloor} \pi_3(2^{l-1})\pi_3(2^l,M)\pi_1(2^l,M) \le &\sum_{l=1}^{\lfloor \log M\rfloor} \pi_3(2^{l-1})\pi_3(2^l,M)\frac{2^{\epsilon l}}{M^\epsilon}\\
\le& C\pi_3(M) \sum_{l=1}^{\lfloor \log M\rfloor} 2^{\epsilon (l-\log M)}. 
\end{align*}
The sum in the final term is 
\begin{align*}
M^{-\epsilon} \sum_{l=1}^{\lfloor \log M\rfloor} 2^{\epsilon l} \le C,
\end{align*}
so 
\[\mathbf{P}(\mathcal{A}_3(e) \mid A_n)\le C\pi_3(\min(\mathrm{dist}(0,e),\mathrm{dist}(e,\partial B_n)).\]

Next, let $\mathcal{A}_4(e) = \{e\in \cup_{k=1}^{K} \mathcal{C}_k\}$. Applying Lemma \ref{lma: circuits} and gluing, and summing over values of $l$, we obtain as in the cases above:
\[\mathbf{P}(\mathcal{A}_4(e) \mid A_n)\le C\pi_3(\mathrm{dist}(e,0))\le C\pi_3(M(e)).\]

We can now return to the probability in \eqref{eqn: labor-union}, and use the estimate:
\begin{align*}
\mathbf{P}(e\in (\cup_{k=1}^K \mathcal{C}_k) \cup (\cup_{k=1}^{K+1}\tilde{\sigma}_n^k)\mid A_n) &\le \sum_{j=1}^4 \mathbf{P}(\mathcal{A}_j(e) \mid A_n)\\
&\le C \pi_3(\min(\mathrm{dist}(0,e),\mathrm{dist}(e,\partial B_n))).
\end{align*}
Summing over the values of $e$ and using Lemma \ref{lem: artem}, we find
\[\mathbf{E}[S_{B_n(0)}; C_0\mid A_n] \le Cn^2\pi_3(n).\]
This completes the proof of Theorem~\ref{thm: radial}.

\section{The unrestricted chemical distance}
\label{sec: chem}
In this section, we prove Corollary \ref{thm: pt2pt} concerning the distance between two points $x$ and $y$ in $\mathbb{Z}^2$, with no spatial restriction on the paths connecting them. As a first step, we consider the case where $x$ and $y$ are connected in a box. More precisely, suppose $x,y\in B_n(0)$; we let $\mathrm{dist}^{B_n}_{chem}(x,y)$ be the least number of edges in any open path connecting $x$ to $y$ inside $B_n(0)$. We  use the same strategy as in the proof of Theorem~\ref{thm: radial} to show 
\begin{prop}\label{prop: superman}
There is a constant $C>0$ such that, for all $n$ and all $x,y\in B_{n/2}(0)$,
\begin{equation}
\mathbf{E}[\mathrm{dist}^{B_n}_{chem}(x,y)\mid x\leftrightarrow_{B_n} y] \le Cn^2\pi_3(n)\label{eqn: pttoptbox}.
\end{equation}
Here, $\{x\leftrightarrow_{B_n} y\}$ is the event that $x$ is connected to $y$ by an open path inside $B_n$.
\end{prop}

It is the estimate \eqref{eqn: pttoptbox} which gives us some control over the tail of the distance between $x$ and $y$, leading to \eqref{eqn: chem}. We assume Proposition \ref{prop: superman} for the moment, and prove Corollary \ref{thm: pt2pt}.

\begin{proof}[Proof of Corollary \ref{thm: pt2pt}]
For $x,y$ given and $m \geq 0$, let $B_m' = B_m(\frac{1}{2}(x+y))$. Define $d=\|x-y\|_\infty$ and let $K$ be the least $K'\ge 1$ such that there is a closed dual circuit in $B_{2^{K+1}d}'\setminus B_{2^K d}'$ around $B_{2^K d}'$. By planarity, if $K\le L$, and $x\leftrightarrow y$, then $x$ is connected to $y$ in $B_{2^{L+1}d}'$.

For $\lambda>0$ and $L \geq 1$, write
\begin{align}\label{eqn: cookie}
\mathbf{P}(\mathrm{dist}_{chem}(x,y)\ge \lambda d^2\pi_3(d) \mid x\leftrightarrow y) &\le \mathbf{P}(\mathrm{dist}_{chem}(x,y)\ge \lambda d^2\pi_3(d), K\le L\mid x\leftrightarrow y)\\
&\quad + \mathbf{P}(K\ge L\mid x\leftrightarrow y). \nonumber
\end{align}
We claim first that there exists a $c>0$ independent of $x$ and $y$ such that
\[\mathbf{P}(K\ge L \mid x\leftrightarrow y) \le Ce^{-cL}.\]
To see this, note that $\{ x\leftrightarrow y\}\subset \{x\leftrightarrow B_{d/4}(x)\}\cap \{y\leftrightarrow B_{d/4}(y)\}$. Furthermore, a simple construction using the Russo-Seymour-Welsh methodology (see Figure \ref{const}) gives
\begin{equation}\label{eq: RSW_argument}
\mathbf{P}(x\leftrightarrow y)\geq C \mathbf{P}(x\leftrightarrow B_{d/4}(x), y\leftrightarrow B_{d/4}(y)).
\end{equation}

\begin{figure}
\centering
\includegraphics[scale = 0.75]{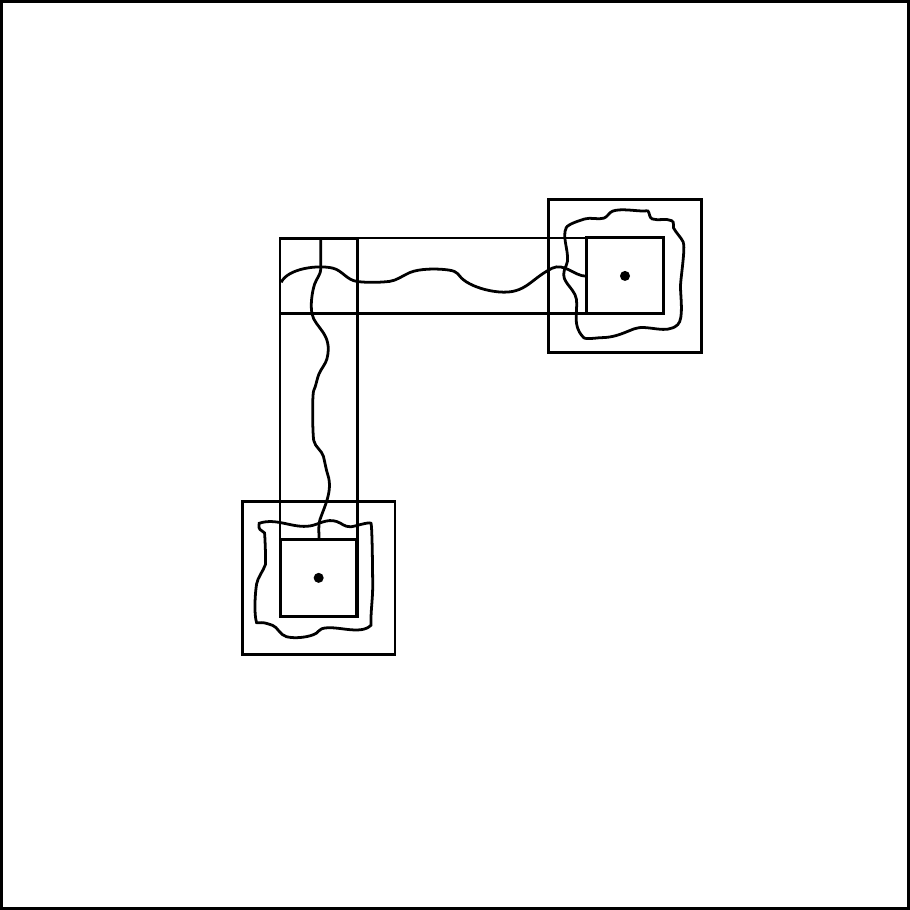}
\caption{Depiction of the RSW-type argument used in the proof of Corollary~\ref{thm: pt2pt}. The points $x$ and $y$ are the centers of the small squares. By gluing together open circuits around $x$ and $y$ in $B_{d/4}(x) \setminus B_{d/8}(x)$ and $B_{d/4}(y) \setminus B_{d/8}(y)$ with the horizontal and vertical open paths that reside in the small rectangles, any open connections from $x$ to $\partial B_{d/4}(x)$ and $y$ to $\partial B_{d/4}(y)$ must be connected to each other. The displayed open paths cost constant probability by RSW, so the FKG inequality gives \eqref{eq: RSW_argument}.}
\label{const}
\end{figure}

Since $\{K \ge L\}$ depends only on edges outside $B_{2d}'$, the above implies
\begin{equation}\label{eqn: decay}
\mathbf{P}(K\ge L\mid x\leftrightarrow y) \le C\mathbf{P}(\cap_{k=1}^{L-1}D_k^c)\le Ce^{-cL},
\end{equation}
where $D_k$ is the event that there is a closed dual circuit around $B_{2^kd}'$ in $B_{2^{k+1}d}'\setminus B_{2^kd}'$.

Returning to \eqref{eqn: cookie}, we have
\begin{align}
&\mathbf{P}(\mathrm{dist}_{chem}(x,y)\ge \lambda d^2\pi_3(d), K\le L\mid x\leftrightarrow y) \\
\le~& \frac{\mathbf{P}(\mathrm{dist}^{B_{2^{L+1}d}'}_{chem}(x,y)\ge \lambda d^2\pi_3(d), x\leftrightarrow_{B'_{2^{L+1}d}} y )}{\mathbf{P}(x\leftrightarrow_{B'_{2^{L+1}d}} y)} \nonumber \\
\le~& \frac{C}{\lambda}\frac{2^{2L}d^2\pi_3(2^Ld)}{d^2\pi_3(d)} \nonumber \\
\le~& \frac{C}{\lambda} 2^{2L}\pi_3(d,2^Ld).\label{eqn:  22L}
\end{align}
Choosing $L=\lfloor \log \lambda^{1/2}\rfloor$, we find by \eqref{eqn: decay} and \eqref{eqn: 22L}
\[\mathbf{P}(\mathrm{dist}_{chem}(x,y)\ge \lambda d^2\pi_3(d))\le \frac{C}{\lambda^{c'}}+\lambda^{-c'},\]
with some (small) $c'>0$.
\end{proof}

We now turn to the proof of the bound \eqref{eqn: pttoptbox}.
\begin{proof}[Proof of Proposition \ref{prop: superman}]
As mentioned previously, the proof is very similar to that of Theorem \ref{thm: radial}. Suppose first that we are on the event $\{x\leftrightarrow y\}\cap (C_0')^c$, where $(C'_0)^c$ is the event that there is a closed dual path in $B_n(0)^*$ from a dual vertex adjacent to $x$ to a dual vertex adjacent to $y$. Equivalently, there is no open path in $B_n(0)$ separating $x$ from $y$.

Then, as in Section \ref{sec: C0c}, there is a path $\tilde{\sigma}_n$ from $x$ to $y$ in $B_n(0)$ such that each edge $e\in \tilde{\sigma}_n$ has two disjoint open arms and one closed dual arm to distance
\[\min(\mathrm{dist}(e,x),\mathrm{dist}(e,y)).\]
In this case, we obtain the result in a manner analogous to Theorem \ref{thm: radial}.
 
On $\{x\leftrightarrow y\} \cap C_0'$, let $K$ be the maximal number of disjoint, possibly closed Jordan curves consisting of open edges in $B_n(0)$ separating $x$ from $y$ in $B_n(0)$. Each such curve separates $B_n(0)$ into two connected components, one containing $x$, and the other containing $y$. 

If a lattice path $\gamma=(e_1,\ldots, e_N)$ of open edges separates $x$ from $y$, then either
\begin{enumerate}
\item $\gamma$ is a circuit ($e_1$ touches $e_N$) around $x$ with $y \notin \mathrm{int}(\gamma)$, or $\gamma$ is  circuit around $y$ with $y\notin\mathrm{int}(\gamma)$. 
\item or both $e_1$ and $e_N$ have an endpoint on $\partial B_n(0)$, and any circuit formed by concatenating $\gamma$ with a portion of $\partial B_n(0)$ between $e_1$ and $e_N$ contains $x$ but not $y$, or contains $y$ but not $x$. In this case, we may assume $\gamma$ touches the boundary only twice.
\end{enumerate}
Given a lattice curve $\gamma$ separating $x$ from $y$ in $B_n(0)$, we denote by $\mathcal{B}_x(\gamma)$ the component of $B_n\setminus \gamma$ containing $x$.

On $C_0$, let $\mathcal{C}_1(x)$ be the path of open edges separating $x$ from $y$ such that $\mathcal{B}_x(\mathcal{C}_1(x))$ is minimal. 
Similarly, $\mathcal{C}_2(x)$ is defined to be the path of open edges lying in $B_n\setminus \mathcal{B}_x(\mathcal{C}_1(x))$ separating $x$ and $y$ such that $\mathcal{B}_x(\mathcal{C}_2(x))\setminus \mathcal{B}_x(\mathcal{C}_1(x))$ is minimal, and so on for $\mathcal{C}_3(x), \ldots, \mathcal{C}_K(x)$. At this point we can transpose the proof of Theorem \ref{thm: radial} almost verbatim, replacing successive innermost circuits by the paths $\mathcal{C}_1(x), \ldots, \mathcal{C}_K(x)$. 
\end{proof}

\section{Non-concentration}
The bound of order $\lambda^{-c}$ for $c>0$ obtained in the previous section is too weak to even bound the expectation. In fact, here we prove that the chemical distance cannot be very concentrated. Let $\mathbf{e}_1=(1,0)$. Then
\begin{equation}\label{moment}
\mathbf{E}[\mathrm{dist}_{chem}(0,\mathbf{e}_1)^2; 0\leftrightarrow \mathbf{e}_1] =\infty.
\end{equation}
Proposition \ref{prop: nomoment} is equivalent to equation \eqref{moment} because it is clear that $\mathbf{P}((0,0)\leftrightarrow (1,0))\ge 1/2$.

To prove \eqref{moment}, let $D_{\mathbf{e}_1}$ be the least $k \geq 1$ such that $(0,0)$ is connected to  $\mathbf{e}_1$ in $B_{2^k}(0)$. We define $D_{\mathbf{e}_1}=\infty$ if there is no such $k$.
The basic observation leading to \eqref{moment} is the following. Put $\mathbf{e} = \{0, \mathbf{e}_1\}$.
\begin{lma}\label{lem: taco_head}
If $D_{\mathbf{e}_1}=k$, then the dual edge $\mathbf{e}^*$ has two disjoint closed dual arms to $\partial B_{2^{k-1}}^*$, and each of $0$ and $\mathbf{e}_1$ has an open arm to $\partial B_{2^{k-1}}$.
\begin{proof}
Since $0$ and $\mathbf{e}_1$ are not connected in $B_{2^{k-1}}(0)$, there is a dual path separating the two edges. This path must include the dual edge $\mathbf{e}^*$. On the other hand, $0$ and $\mathbf{e}_1$ are connected in $B_{2^k}$, but not $B_{2^{k-1}}$ so there is an open path with endpoints $0$ and $\mathbf{e}_1$ which contains an edge in $B_{2^k}\setminus B_{2^{k-1}}$.
\end{proof}
\end{lma}

As a corollary, we obtain
\begin{cor}
\begin{equation}
\mathbf{P}(D_{\mathbf{e}_1}=k)\asymp \pi_4(2^{k-1}),\label{eqn: Debd}
\end{equation}
where $\pi_4(\ell)$ is the probability that there are four alternating arms from $\mathbf{e}$ to $\partial B_\ell$.
\begin{proof}
The upper bound follows directly from Lemma~\ref{lem: taco_head}.

For the lower bound, suppose that $\mathbf{e}^*$ has two disjoint closed arms to $\partial B_{2^{k-1}}$, 0 has one open arm to the left side of $\partial B_{2^{k-1}}$, and $\mathbf{e}_1$ has one open arm to the right side of $\partial B_{2^{k-1}}$. Denote this event by $E_k$. By \cite[Theorem 11]{nolin}, we know that
\[\mathbf{P}(E_k)\asymp \pi_4(2^{k-1}).\]
On $E_k\cap \{(0,\mathbf{e}_1) \text{ is closed}\}$, we have $D_{\mathbf{e}_1}\ge k$. On $E_k$, we can use the generalized FKG lemma \cite[Lemma 13]{nolin} and a RSW construction as in Figure \ref{construction} to ensure that $\{0\leftrightarrow_{B_{2^k}} \mathbf{e}_1\}$ occurs:
\begin{equation}\label{eq: taco_bell_inequality}
\mathbf{P}(0\leftrightarrow_{B_{2^k}} \mathbf{e}_1, E_k)\ge C\mathbf{P}(E_k).
\end{equation}
Since $D_{\mathbf{e}_1}\le k$ on $\{0\leftrightarrow_{B_{2^k}} \mathbf{e}_1\}$, we have
\[\mathbf{P}(D_{\mathbf{e}_1}=k)\ge C\pi_4(2^{k-1}).\]

\end{proof}
\end{cor}

\begin{figure}
\centering
\includegraphics[scale = 0.40]{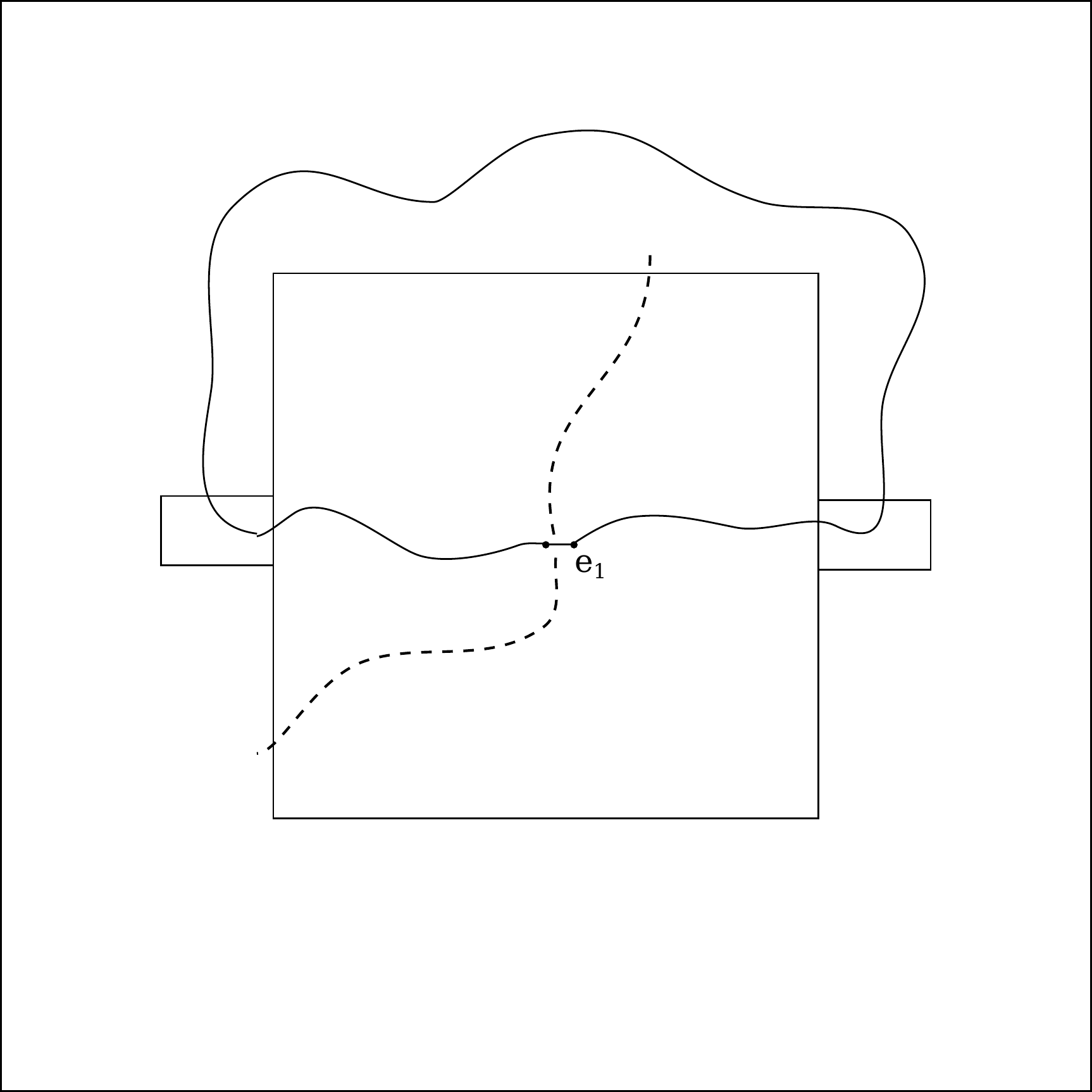}
\caption{Depiction of the RSW construction used in the proof of \eqref{eq: taco_bell_inequality}. The four arms from $\mathbf{e}$ are well-separated in the sense that the open arms touch opposite sides of the box $B_{2^{k-1}}$. This allows one to use the generalized FKG inequality and RSW to extend these arms in $B_{2^k}$ around the interior box to connect and ensure that $D_{\mathbf{e}_1} \leq k$.}
\label{construction}
\end{figure}

We can now give the proof of \eqref{moment}. 

\begin{proof}[Proof of Proposition \ref{prop: nomoment} ] On $\{D_{\mathbf{e}_1}=k\}$, we have $\mathrm{dist}_{chem}(0,\mathbf{e}_1)\ge 2^{k-1}$, so
\begin{align*}
\mathbf{E}[\mathrm{dist}_{chem}(0,\mathbf{e}_1)^2; D_{\mathbf{e}_1}<\infty] &\ge \sum_{k=1}^\infty 2^{2(k-1)}\mathbf{P}(D_{\mathbf{e}_1}=k)\\
&\ge C \sum_{k=1}^\infty 2^{2k}\pi_4(2^k).
\end{align*}
The four arm probability $\pi_4(2^k)$ is bounded below by the five-arm probability, for which we have a ``universal'' asymptotic \cite[Theorem 24, 2.]{nolin}:
\[\pi_4(2^k)\ge \pi_5(2^k)\asymp  2^{-2k},\]
from which it is manifest that the sum is divergent.
\end{proof}

\section{Quantitative bounds on $S_n$}
\label{sec: quantitative}
In this final section, we derive Theorem \ref{thm: quantitative}. Unlike in the previous sections, we will rely on the framework in our paper \cite{DHS15}. 

We begin by summarizing the strategy introduced there and the relevant facts we need for our proof, referring to \cite{DHS15} for details. Then, in Section \ref{sec: quantest}, we show how a modification of the strategy of \cite{DHS15} leads to the improved estimate \eqref{eqn: ELn}. Sections \ref{sec: newevent} and \ref{sec: Ek} contain the main new idea for the approach we take here: a new definition of events which forces the existence of a shortcut around a portion of the lowest crossing. The main advantage of this definition is that the probability of the events is much higher than those in \cite{DHS15} --- this probability is estimated in Sections \ref{sec: neweventprob} and \ref{sec: neweventprob2}. 
The proof of Corollary~\ref{cor: quantitative} follows from Markov's inequality and the right choice of parameters in an estimate from \cite{DHS15} on the lower tail of the distribution of $L_n$. This is explained in the final Section \ref{sec: lowertail}.

We first recall some notation. On $H_n$ (defined in \eqref{eqn: Hn}), any self-avoiding path $\gamma$ with one endpoint on the left side $\{-n\}\times [-n,n]$ of $B_n(0)$ and the other on the right side is a Jordan arc separating the top side from the bottom side. The connected component of $[-n,n]^2 \setminus \gamma$ connected to the bottom side is denoted $\mathcal{B}(\gamma)$. The lowest crossing $l_n$ of $B_n(0)$ is the horizontal crossing $\gamma$ such that the region $\mathcal{B}(\gamma)$ is minimal. We denote this region by $\mathcal{B}(l_n)$.

The following proposition shows that $\pi_3(n)\ge n^{-1+\alpha_3}$ for some $\alpha_3>0$. 
\begin{prop}{\cite[Lemma 4]{DHS15}}\label{prop: boss}
For some $C,C'>0$ and $\alpha_3 \in (0,1)$, 
\[C (m/n)^{1+\alpha_3} \leq (n/m)^2 \pi_3(m,n) \text{ for all } 1 \leq m \leq n,\] or
\begin{equation}\label{eqn: aizenmanslegacy}
\pi_3(m,n) \geq C' (n/m)^{\alpha_3-1} \text{ for all } 1 \leq m \leq n.
\end{equation}
Here, $\pi_3(m,n)$ is the probability that there are two disjoint open paths connecting $B(m)$ to $\partial B(n)$ and a closed dual path connecting $B(m)^*$ to $\partial B(n)^*$.
\end{prop}

The significance of this lower bound for us is that the typical order of the length $L_n = \# l_n$ of the lowest crossing, $n^2\pi_3(n)$, is more than linear. Thus we can exclude a region around $\partial B_n$ of width $n^{\alpha_3/2}$ from consideration, as it is irrelevant for counting the length. Define:
\begin{equation}
\begin{split}
\Lambda_n &= B_{n-n^{\alpha_3/2}}(0)\\
&= [-n+n^{\alpha_3/2}, n-n^{\alpha_3/2}]^2.
\end{split}
\end{equation}

We denote by $\hat{l}_n$ the intersection of $l_n$ with $\Lambda_n$. The key definition \cite[Definition 5]{DHS15} of our approach is that of a $\epsilon$-shielded detour. An important difference from \cite{DHS15} is that we will choose $\epsilon$ to be $n$-dependent. See Proposition \ref{prop: probest2} below.

\begin{df}[$\epsilon$-shielded detour]\label{def: edetour}
For an edge $e\in \hat{l}_n$, the set $\mathcal{S}(e,\epsilon)$ of \emph{$\epsilon$-shielded detours} around $e$ is defined as the set of self-avoiding paths $P$ with vertex set $w_0,\ldots, w_M$ such that
\begin{enumerate}
\item for $i=1,\ldots M-1$, $w_i\in [-n,n]^2\setminus \mathcal{B}(l_n)$,
\item the edges $\{w_0, w_0 + \mathbf{e}_1\}$, $\{w_0-\mathbf{e}_1,w_0\}$, $\{w_M,w_M+\mathbf{e}_1\}$, and $\{w_M-\mathbf{e}_1,w_M\}$ are in $l_n$ and $w_1 = w_0+\mathbf{e}_2$, $w_{M-1}= w_M+\mathbf{e}_2$.
\item writing $Q$ for the subpath of $l_n$ from $w_0$ to $w_M$ containing $e$, the path $Q\cup P$ is a closed circuit in $B_n(0)$,
\item The points $w_0 + (1/2)(-\mathbf{e}_1+\mathbf{e}_2)$ and $w_M + (1/2)(\mathbf{e}_1 + \mathbf{e}_2)$ are connected by a dual closed self-avoiding path $R$, whose first and last edges are vertical (translates of $\{0,\mathbf{e}_2\}$), lying in $B_n(0)\setminus \mathcal{B}(l_n)$.
\item $\#P\leq \epsilon \# Q$.
\end{enumerate}
\end{df}

Given some fixed deterministic ordering of all finite lattice paths, we define $\pi(e)=\pi(e,\epsilon)$ to be the first element of $\mathcal{S}(e,\epsilon)$ in this ordering. If $\mathcal{S}(e,\epsilon)$ is empty, then we set $\pi(e)=\emptyset$.

The collection of detours $(\pi(e):e\in \hat{l}_n)$ has the following properties. The proofs are found in \cite[Section 6]{DHS15}:
\begin{enumerate}
\item For distinct $e$, $e'\in \hat{l}_n$, $\pi(e)$ and $\pi(e')$ are either equal or vertex-disjoint.
\item If $e\in \hat{l}_n$ and $\pi(e)\neq \emptyset$ with vertices $w_0, \ldots, w_M$ as above, then $w_0,w_M\in l_n$, but $w_i\in [-n,n]^2\setminus\mathcal{B}(l_n)$, for $i=1,\ldots, M-1$.
\item If $e\in \hat{l}_n$, let $\hat{\pi}(e)$ be the segment of $l_n$ from $w_0$ to $w_M$ containing $e$ (the ``detoured portion'' of $l_n$). Then
\[\#\pi(e)\le \epsilon \#\hat{\pi}(e).\]
\end{enumerate}

Our contribution here is the following
\begin{prop} \label{prop: probest2}
Let $0 \le c_2<1/4$ and $\epsilon = 1/(\log n)^{c_2}$.
We have
\begin{equation}\label{eqn: probest}
\mathbf{P}(\pi(e)=\emptyset \mid e\in l_n)=o(1/\log n) \text{ as } n \to \infty
\end{equation}
uniformly in $e\in \Lambda_n$.
\end{prop}

For the remainder of the section, we identify paths with their edge sets. Given detour paths $\pi(e)$, and corresponding detoured subpaths $\hat{\pi}(e)$, we construct crossing $\sigma_n$ which is shorter than $l_n$. We choose a subcollection $\Pi$ of $\{\pi(e):e\in \hat{l}_n\}$ that is maximal in the sense that for any $\pi(e),\pi(e')\in \Pi$, $e\neq e'$, the paths $\hat{\pi}(e)$ and $\hat{\pi}(e')$ are vertex disjoint, and the total length of the detoured paths $\sum_{\pi\in \Pi} \#\hat{\pi}$ is maximal. We also write
\[\hat{\Pi} =\{\hat{\pi}: \pi \in \Pi\}.\] 
The crossing $\sigma_n$ is the path with edge set equal to the union of $\Pi$ and the edges in $l_n$ that are not in $\hat{\Pi}$.

The following two lemmas from \cite[Section 3.2]{DHS15} are used in the next subsection to compare $\sigma_n$ to $l_n$:
\begin{lma}
On $H_n$, $\sigma_n$ is an open horizontal crossing of $[-n,n]^2$.
\end{lma}
\begin{lma}\label{lma: logic}
On $H_n$, if $e\in \hat{l}_n\setminus \hat{\Pi}$, then $\pi(e)=\emptyset$.
\end{lma}

\subsection{Estimate for $S_n$}\label{sec: quantest}
We show how Proposition \ref{prop: probest2} is used to obtain Theorem \ref{thm: quantitative}. The argument is essentially the same as in \cite[Section 4]{DHS15}.
\begin{proof}[Proof of Theorem \ref{thm: quantitative}]
Recall the definition  of $\sigma_n$ at the end of the previous section. By this definition and properties 1-3 above (below Definition~\ref{def: edetour}), we have
\begin{align}
\# \sigma_n &= \sum_{\pi\in \Pi} \# \pi + \#(l_n \setminus \cup_{\hat \pi \in \hat \Pi} \hat \pi)\nonumber\\
&\le \sum_{\pi\in \Pi} \# \pi + \#\{e\in B_n(0): e\cap( B_n(0)\setminus \Lambda_n)\neq \emptyset\}+ \#(\hat{l}_n\setminus \cup_{\hat \pi \in \hat \Pi} \hat \pi)\\
&\le \sum_{\pi\in \Pi} \# \pi + 30n^{1+\alpha_3/2} +  \#\{e\in \hat{l}_n: \pi(e)=\emptyset\} \label{eqn: 2ndeq}
\end{align}
To pass to the final inequality \eqref{eqn: 2ndeq}, we have used Lemma \ref{lma: logic} for the third term. 

We can now estimate the length $S_n$ of the shortest crossing(s):
\begin{equation}
\label{eqn: expest} 
\begin{split}
\mathbf{E}S_n &\le \mathbf{E}\# \sigma_n\\
&\le 30n^{1+\alpha_3/2}+\mathbf{E}\# \{e\in \hat{l}_n: \pi(e)=\emptyset\} + \epsilon \cdot \mathbf{E}L_n\\
&\le \frac{C}{(\log n)^{c_2}}n^2\pi_3(n) + \mathbf{E}\# \{e\in \hat{l}_n: \pi(e)=\emptyset\}.
\end{split}
\end{equation}

Using \eqref{eqn: probest}, we can estimate the last expectation:
\begin{align*}
\mathbf{E}\# \{e\in \hat{l}_n: \pi(e)=\emptyset\} &= \sum_{e\in \Lambda(n)}\mathbf{P}(\pi(e)=\emptyset\mid e\in \hat{l}_n)\mathbf{P}(e\in \hat{l}_n)\\
&= o\left(\frac{1}{\log n}\right) \cdot  \mathbf{E}L_n\\
&= o\left(\frac{1}{(\log n)^{c_2}}\right)\cdot n^2\pi_3(n).
\end{align*}
Using this bound in \eqref{eqn: expest}, we obtain \eqref{eqn: ELn}.
\end{proof}

\subsection{A new definition of the events $E_k$}\label{sec: newevent}
Let $K=K(\epsilon,k)$ be the least $l\in \mathbb{Z}_+$ such that
\[2^l \ge (1/\epsilon)\cdot 2^k.\]
We define an event $E_k=E_k(\epsilon)$ which implies the existence of a shortcut inside the annulus $B(2^K)\setminus B(2^k)$, and depends only on the status of edges inside that annulus. Like the event $E_k$ in \cite[Section 5]{DHS15}, the precise definition of the events $E_k$ we use here is involved.

We denote by $E_k(e) = E_k(\epsilon,e)$ the event $\tau_{-e_x}E_k(\epsilon)$, that is, the event that $E_k$ occurs in the configuration $(\omega_{e+e_x})_{e\in\mathcal{E}(\mathbb{Z}^2)}$ translated by the coordinates of the lower-left endpoint $e_x$ of the edge $e$. 

The two essential properties of $E_k(e)$ are 
\begin{enumerate}
\item If $E_k(\epsilon,e)$ occurs for some $k \leq n^{\alpha_3/4}$ and $e$ lies on $\hat l_n$, then there is an $\epsilon$-shielded detour around $e$, in the sense of Definition \ref{def: edetour}. See Lemma \ref{lma: ekprime}. 
\item We have the following lower bound for the probability of $E_k$:
\[\mathbf{P}(E_k\mid A_3(n^{\alpha_3/2}))\ge C\epsilon^4.\]
This is implied by Proposition \ref{prop: probest}.
\end{enumerate}

The next subsection contains an enumeration of the conditions for the event $E_k$ to occur in $B(2^K)\setminus B(2^k)$. The essential features of the construction are as follows:

\begin{enumerate}
\item An open arc (detour) connecting two arms emanating from the 3-arm edge $e$. This arc lies inside a box of sidelength of order $2^k$, and is depicted as the top arc in Figure~\ref{rswf}.
\item A subsegment of the lowest crossing of $B_{n^{\alpha_3/2}}$, of length on the order of $2^{2K}\pi_3(2^K)$. This path is depicted as the pendulous curve in Figure~\ref{rsw2}.
\end{enumerate}
\subsection{Definition of $E_k$}
\label{sec: Ek}

First, we have the following conditions depending on the status of edges inside $[-3\cdot 2^k,3\cdot 2^k]^2$ (see Figure \ref{rswf}). In item 4 (and in item 3, the mirror image), we use the term five-arm point in the following way. The origin (for example) is a five-arm point if it has three disjoint open paths emanating from 0, one taking the edge $\{0, \mathbf{e}_1\}$ first, one taking the edge $\{0, -\mathbf{e}_1\}$ first, and one taking the edge $\{0, \mathbf{e}_2\}$ first, and directed analogously to that in item 4 (or item 3). The two remaining closed dual paths emanate from dual neighbors of 0, one taking the dual edge $\{(-1/2)\mathbf{e}_1+(1/2)\mathbf{e}_2, (-1/2)\mathbf{e}_1 + (3/2)\mathbf{e}_2\}$ first, and the other taking the dual edge $\{(1/2)\mathbf{e}_1-(1/2)\mathbf{e}_2, (1/2)\mathbf{e}_1-(3/2)\mathbf{e}_2\}$ first.

\begin{enumerate}
\item There is a horizontal open crossing of $[2^k,3\cdot 2^k] \times [-\frac{2^k}{3}, \frac{2^k}{3}] $, a horizontal open crossing of $[-\frac{7}{3}\cdot 2^k, -2^k]\times[-\frac{2^k}{3}, \frac{2^k}{3}]$.
\item There is a vertical open crossing of $[-\frac{7}{3}\cdot 2^k, -\frac{5}{3}\cdot 2^k]\times [-3\cdot 2^k, \frac{1}{3}\cdot 2^k]$.

\item There is a five-arm point in the box $[\frac{7}{3}\cdot 2^k, 3\cdot 2^k]\times [-\frac{1}{3}\cdot 2^k, \frac{1}{3}\cdot 2^k]$: in clockwise order: 
\begin{enumerate}
\item an open arm connected to $[\frac{7}{3}\cdot 2^k, \frac{8}{3}\cdot 2^k]\times \{\frac{1}{3}\cdot 2^k\}$, 
\item a closed dual arm connected to $[\frac{8}{3}\cdot 2^k, 3\cdot 2^k]\times \{\frac{1}{3}\cdot 2^k\}$, 
\item an open arm connected to the ``right side'' of the box $\{3\cdot 2^k\}\times [-\frac{1}{3}\cdot 2^k, \frac{1}{3}\cdot 2^k]$, 
\item a closed dual arm connected to $[\frac{7}{3}\cdot 2^k, 3\cdot 2^k]\times \{-\frac{1}{3}\cdot 2^k\}$, 
\item and an open arm connected to $\{\frac{7}{3}\cdot 2^k\}\times [-\frac{1}{3}\cdot 2^k,\frac{1}{3}\cdot 2^k]$.
\end{enumerate}

\item There is a five-arm point in the box $[-3\cdot 2^k, -\frac{8}{3}\cdot 2^k ]\times [-\frac{1}{3}\cdot 2^k, \frac{1}{3}\cdot 2^k]$, in clockwise order:
\begin{enumerate}
\item a closed dual arm connected to $[-3\cdot 2^k, -\frac{8}{3}\cdot 2^k]\times \{\frac{1}{3}\cdot 2^k\}$,
\item an open arm connected to $[-\frac{8}{3}\cdot 2^k, -\frac{7}{3}\cdot 2^k]\times \{\frac{1}{3}\cdot 2^k\}$,
\item an open arm connected to  $[-\frac{8}{3}\cdot 2^k, -\frac{7}{3}\cdot 2^k]\times \{-\frac{1}{3}\cdot 2^k\}$,
\item a closed dual arm connected to $[-3\cdot 2^k, -\frac{8}{3}\cdot 2^k]\times \{-\frac{1}{3}\cdot 2^k\}$,
\item and an open arm connected to the ``left side'' of the box, $\{-3\cdot 2^k\}\times [-\frac{1}{3}\cdot 2^k, \frac{1}{3}\cdot 2^k]$.
\end{enumerate}

\item There is an open vertical crossing of $[-3\cdot 2^k, -\frac{7}{3}\cdot 2^k]\times [-3\cdot 2^k,-\frac{1}{3}\cdot 2^k]$, connected to the open arm emanating from the five-arm point in $[-3\cdot 2^k, -\frac{8}{3}\cdot 2^k ]\times [-\frac{1}{3}\cdot 2^k, \frac{1}{3}\cdot 2^k]$ landing in $[-\frac{8}{3}\cdot 2^k, -\frac{7}{3}\cdot 2^k]\times \{-\frac{1}{3}\cdot 2^k\}$. There is a dual closed vertical crossing of $[-3\cdot 2^k, -\frac{7}{3}\cdot 2^k]\times [-3\cdot 2^k, -\frac{1}{3}\cdot 2^k]$, connected to the closed dual arm in $[-3\cdot 2^k, -\frac{8}{3}\cdot 2^k ]\times [-\frac{1}{3}\cdot 2^k, \frac{1}{3}\cdot 2^k]$ landing in  $[-3\cdot 2^k, -\frac{8}{3}\cdot 2^k]\times \{-\frac{1}{3}\cdot 2^k\}$.

\item There is a closed dual vertical crossing of $[\frac{7}{3}\cdot 2^k, 3\cdot 2^k]\times [-3\cdot 2^k, -\frac{1}{3}\cdot 2^k]$, connected to the dual arm in $[\frac{7}{3}\cdot 2^k, 3\cdot 2^k]\times [-\frac{1}{3}\cdot 2^k, \frac{1}{3}\cdot 2^k]$.

\item There is a closed dual arc (the shield) in the half-annulus $[-3\cdot 2^k, 3\cdot 2^k]^2 \setminus (-\frac{8}{3} \cdot 2^k, \frac{8}{3} \cdot 2^k)^2 \cap \{(x,y) : y \geq -\frac{1}{3} \cdot 2^k\}$ connecting the closed dual paths from the two five-arm points in items 3 and 4.

\item There is an open arc (the detour) in the half-annulus $[-\frac{8}{3} \cdot 2^k, \frac{8}{3} \cdot 2^k]^2 \setminus (-\frac{7}{3} \cdot 2^k, \frac{7}{3} \cdot 2^k)^2 \cap \{(x,y) : y \geq -\frac{1}{3} \cdot 2^k\}$ connecting the open paths from the two five-arm points in items 3 and 4 which end on the line $\{(x,\frac{1}{3} \cdot 2^k) : x \in \mathbb{Z}\}$.
\end{enumerate}

Next, we have the following constructions around the box $[-2^K,2^K]^2$;
see Figure \ref{rsw2}.
\begin{enumerate}
\setcounter{enumi}{8}
\item There is a closed dual arc in $B((0,-2^K/2),2^K/4)\setminus B((0,-2^K/2),2^K/8)$ around 
\[B((0,-2^K/2),2^K/8)),\] joining $[-\frac{2^K}{8},-\frac{2^K}{16}]\times \{-\frac{2^K}{4}\}$ to $[\frac{1}{16}\cdot 2^K, \frac{1}{8}\cdot 2^K]\times\{- \frac{2^K}{4}\}$.
\item There are two disjoint closed paths inside $[-\frac{2^K}{8},\frac{2^K}{8}]\times ((-\infty,0])$: one joining the endpoint of the vertical closed crossing on $[-3\cdot 2^k, -\frac{8}{3}\cdot 2^k]\times \{-3\cdot 2^k\}$ to the endpoint of the closed dual arc in the previous item on $[-\frac{2^K}{8},-\frac{2^K}{16}]\times \{-\frac{2^K}{4}\}$, the second, joining the endpoint of the vertical crossing on $[\frac{7}{3}\cdot 2^k,3\cdot 2^k]\times \{-3\cdot 2^k\}$, to the endpoint of the closed dual arc in the previous item on $[\frac{2^K}{16},\frac{2^K}{8}]\times \{-\frac{2^K}{4}\}$.
\item There is a horizontal open crossing of the square box $[-\frac{1}{8}\cdot 2^K, \frac{1}{8}\cdot 2^K]\times [-\frac{5}{8}\cdot 2^K, -\frac{3}{8}\cdot 2^K]$
\item There are vertical open crossings of $[-\frac{1}{8}\cdot 2^K, -\frac{1}{16}\cdot 2^K]\times [-\frac{5}{8}\cdot 2^K, -\frac{3}{8}\cdot 2^K]$ and $[\frac{1}{16}\cdot 2^K, \frac{1}{8}\cdot 2^K]\times [-\frac{5}{8}\cdot 2^K, -\frac{3}{8}\cdot 2^K]$.

\item There are two disjoint open paths contained in $[-\frac{1}{16}\cdot 2^K,\frac{1}{16}\cdot 2^K]\times [-\frac{1}{4}\cdot 2^K,0] \cup [-\frac{2^K}{8},\frac{2^K}{8}]\times [-\frac{3}{8}\cdot 2^K, -\frac{1}{4}\cdot 2^K]$:
\begin{enumerate} 
\item one joining the endpoint of the  open vertical crossing of $[-\frac{8}{3}\cdot 2^k, -\frac{7}{3}\cdot 2^k]\times [-3\cdot 2^k,-\frac{1}{3}\cdot 2^k]$  to the endpoint of the open crossing of $[-\frac{1}{8}\cdot 2^K, -\frac{1}{16}\cdot 2^K]\times [-\frac{5}{8}\cdot 2^K, -\frac{3}{8}\cdot 2^K]$ on $[-\frac{1}{8}\cdot 2^K, -\frac{1}{16}\cdot 2^K]\times \{-\frac{3}{8}\cdot 2^K\}$, 
\item one joining the endpoint of the open vertical crossing of $[-\frac{7}{3}\cdot 2^k,-\frac{5}{3}\cdot 2^k ]\times [-3\cdot 2^k,-\frac{1}{3}\cdot 2^k]$ to  $[\frac{1}{16}\cdot 2^K, \frac{1}{8}\cdot 2^K]\times \{-\frac{3}{8}\cdot 2^K\}$.
\end{enumerate}
\item There is dual closed vertical crossing of $[-\frac{1}{16}\cdot 2^K, \frac{1}{16}\cdot 2^K]\times [-2^K, -\frac{5}{8}\cdot 2^K]$.
\end{enumerate}

Finally, we finish the description of the event, completing the construction by adding two more macroscopic conditions:
\begin{enumerate}\setcounter{enumi}{14}
\item There is a dual closed circuit with two open defects around the origin in $[-2^K, 2^K]^2\setminus [-\frac{7}{8}\cdot 2^K,\frac{7}{8}\cdot 2^K]^2$. One of the defects is contained in $[-2^K, -\frac{7}{8}\cdot 2^K]\times [-\frac{2^K}{8},\frac{2^K}{8}]$, and the other in $[\frac{7}{8}\cdot 2^K, 2^K]\times [-\frac{2^K}{8},\frac{2^K}{8}]$.
\item There are two disjoint open arms: one from the left side $\{-3\cdot 2^k\}\times [-3\cdot 2^k, 3\cdot 2^k]$ of the box $[-3\cdot 2^k, 3\cdot 2^k]^2$ (touching the open arm from the five-arm point that lands there) to the left side of $[-2^K,2^K]^2$, the other from the right side of $\{3\cdot 2^k\}\times [-3\cdot 2^k, 3\cdot 2^k]$ (touching the corresponding open arm from the five-arm point there) to the right side of $[-2^K,2^K]^2$.
\end{enumerate}

\begin{figure}
\centering
\scalebox{0.6}{\includegraphics[trim={0 0 0 0},clip]{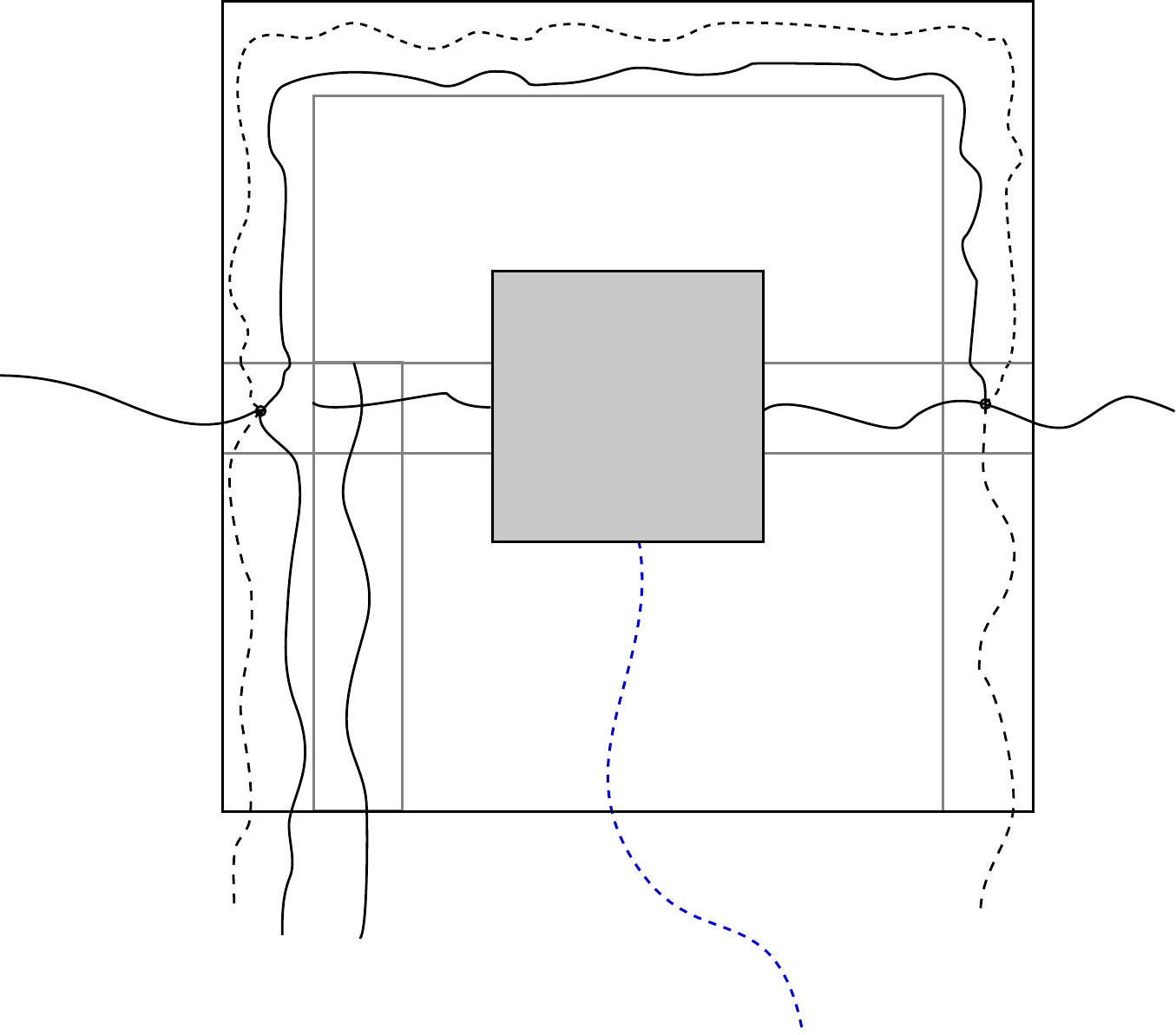}}
\caption{Depiction of the connections in the inner box $[-3 \cdot 2^k, 3 \cdot 2^k]^2$ in the definition of $E_k$. The detour is the solid path connecting the two five-arm points around the top of the box, and the shield is the dotted path just outside it. The detoured path is the solid path starting from the left five-arm point which exits the box in the bottom, goes up to scale $2^K$ in the subsequent constructions, and reenters the box to finish just to the right of where it started. The other dotted connections are closed dual paths that ensure (with the connections on scale $2^K$) that when the center of the shaded box is on the lowest crossing of $B_n$, so are the two five-arm points.}
\label{rswf}
\end{figure}

\begin{figure}
\centering
\scalebox{0.5}{\includegraphics[trim={0 0 0 2cm},clip]{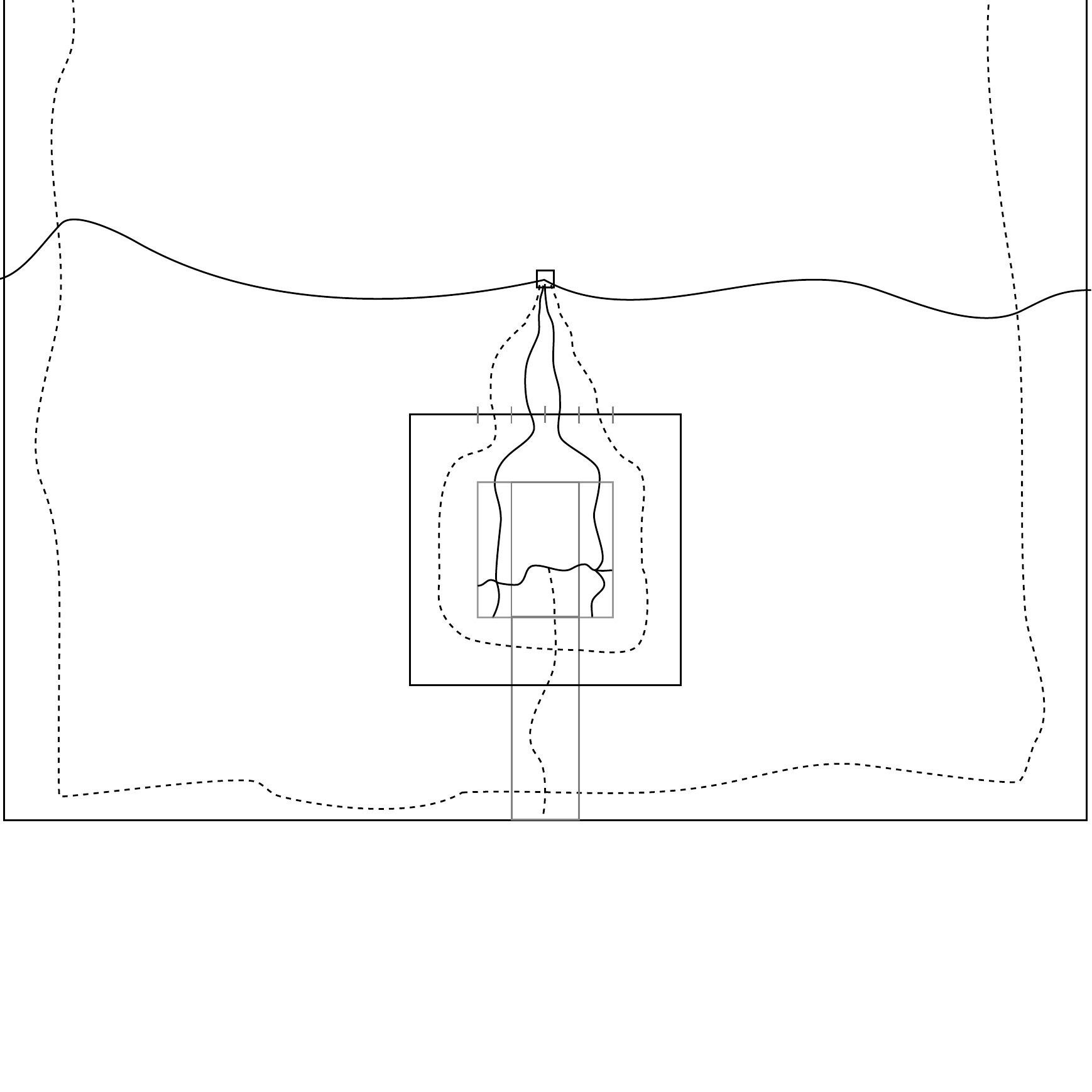}}
\caption{Depiction of connections 9-16 in the definition of $E_k$. The outermost box is $[-2^K,2^K]^2$. The dangling path is a portion of the ``detoured path'' originating in the smaller box $[-3\cdot 2^k, 3\cdot 2^k]^2$. The dotted paths are closed dual paths which ensure that when the center of the smallest box is on the lowest crossing of $B_n$, so is the dangling path.}
\label{rsw2}
\end{figure}

\subsection{Volume estimates}
On the event $E_k$, let $N_K$ be the the number of edges inside the box $[-\frac{1}{16}\cdot 2^K, \frac{1}{16}\cdot 2^K]\times [-\frac{9}{16}\cdot 2^K, -\frac{7}{16}\cdot 2^K]$ which lie on a horizontal open crossing of the box $[-\frac{1}{8} \cdot 2^K, \frac{1}{8} \cdot 2^K] \times [ - \frac{5}{8} \cdot 2^K, - \frac{3}{8} \cdot 2^K]$, and whose dual edge is connected by a closed dual path to the vertical crossing of $[-\frac{1}{16}\cdot 2^K, \frac{1}{16}\cdot 2^K]\times [-2^K, -\frac{5}{8}\cdot 2^K]$ in the definition of $E_k$. A second-moment calculation as in \cite[Proposition 16]{DHS15} gives the following. Recall that $K$ is of order $k + \log (1/\epsilon)$.
\begin{lma}\label{lma: volume}
There exists $c>0$ such that for all $\epsilon>0$ and $k$ sufficiently large (independent of $\epsilon$),
\[\mathbf{P}(N_K \ge c 2^{2K}\pi_3(2^K) \mid E_k)\ge c.\]
\end{lma}
Similarly, on $E_k$, let $\ell_k$ be the outermost open arc in 
\begin{align*}
&[-3\cdot 2^k,-\frac{7}{3}\cdot 2^k]\times[-\frac{1}{3}\cdot 2^k, 3\cdot 2^k]\\
\cup&\,[-3\cdot 2^k, 3\cdot 2^k]\times[\frac{7}{3}\cdot 2^k, 3\cdot 2^k]\\ 
\cup&\,[\frac{7}{3}\cdot 2^k, 3\cdot 2^k]\times [-\frac{1}{3}\cdot 2^k, 3\cdot 2^k]
\end{align*}
 joining the two five arm points in the boxes $[\frac{7}{3}\cdot 2^k, 3\cdot 2^k]\times [-\frac{1}{3}\cdot 2^k, \frac{1}{3}\cdot 2^k]$ and $[-3\cdot 2^k, -\frac{8}{3}\cdot 2^k ]\times [-\frac{1}{3}\cdot 2^k, \frac{1}{3}\cdot 2^k]$. Then we have the following lemma. The proof  is similar but simpler than the proof of \cite[Lemma 15]{DHS15}. For the remainder of the paper, we choose $c'$ (in the lemma below) large enough that $\eta < c/2$, where $c$ is from Lemma~\ref{lma: volume}.
\begin{lma}
For any $\eta>0$, there exists $c'$ such that for all $\epsilon>0$ and $k$ large (independent of $\epsilon$)
\[\mathbf{P}(\# \ell_k \le c' 2^{2k}\pi_3(2^k) \mid E_k)\ge 1-\eta.\]
\end{lma}

We now come to the main lemma of this section.
\begin{lma}\label{lma: ekprime}
For $\epsilon>0$, define 
\begin{equation}\label{eqn: ekpdef}
E_k'(e,\epsilon)=E_k(e,\epsilon)\cap \tau_{-e}\{N_K \ge c 2^{2K}\pi_3(2^K)\} \cap \tau_{-e}\{\# \ell_k \le c' 2^{2k}\pi_3(2^k)\}
\end{equation}
For all $\epsilon>0$ sufficiently small and $k \leq n^{\alpha_3/4}$, on the event $E_k'(e,\epsilon)\cap \{e\in \hat{l}_n\}$, there is an $\epsilon$-shielded detour around $e$, in the sense of Definition \ref{def: edetour}.
\begin{proof}
On the event $E_k$, by using the outermost arc $\ell_k$, we find a detour of length at most 
\[\#\ell_k \le c' 2^{2k}\pi_3(2^k)\]
around a portion of the lowest crossing of length at least
\[N_K \ge  c 2^{2K}\pi_3(2^K).\]
The ratio of the length of the detour to the length of the portion of the lowest crossing is thus bounded above by
\begin{align*}
C\frac{2^{2k}\pi_3(2^k)}{(1/\epsilon)^2 2^{2k}\pi_3(\epsilon^{-1}2^k)}&\le C \frac{c'}{c} \epsilon^2 \frac{1}{\pi_3(2^k,2^K)}\\
& <  \epsilon,
\end{align*}
provided $\epsilon$ is small enough. In the last step we have used $\pi_3(m,n)\ge C (m/n)^{1-\alpha_3}$ (Proposition \ref{prop: boss}).

Following the proof of \cite[Proposition 9]{DHS15}, one checks that the outermost arc $\ell_k$ is an $\epsilon$-shielded detour; that is, on $E_k'\cap \{e\in \hat{l}_n\}$, $\ell_k$ satisfies Definition \ref{def: edetour}.
\end{proof}
\end{lma}

\subsection{Conditional probability of $E_k'$}\label{sec: neweventprob}
Recall the definition of the event $E_k'$ in \eqref{eqn: ekpdef}. We provide a lower bound on the probability of $E_k'$. 
\begin{prop}\label{prop: probest} There is a constant $C>0$ such that, for any $n^{\alpha_3 /8} \le k\le n^{\alpha_3 /4}$,
\[\mathbf{P}(E_k'\mid A_3(n^{\alpha_3/2}))\ge C\epsilon^4.\]

\begin{proof}
From Lemma \ref{lma: volume}, the definition of $E_k$, independence of disjoint regions, standard constructions using Russo-Seymour-Welsh, and generalized FKG \cite[Lemma 13]{nolin}, it follows that
\[\mathbf{P}(E_k'\mid A_3(n^{\alpha_3/2})) \ge C \mathbf{P}(A_6(2^{k+2},2^{K-2})),\]
where $A_6(2^{k+2}, 2^{K-2})$ is the probability that there are four open arms, and two closed dual arms from $\partial B_{2^{k+2}}(0)$ to $\partial B_{2^{K-2}}(0)$. This is because the conditions in the definition of $E_k$ (Section \ref{sec: Ek}) that contribute a factor to the probability of $E_k'$ which is not constant relative to $2^k/2^K$ are the ``macroscopic connections'' between the box $[-3\cdot 2^k,3\cdot 2^k]^2$ and boxes of size of order $2^K$: the two closed dual connections in item 10., the two open paths in item 13., and the open arms in item 16. 

Next, we have
\begin{align*}
\mathbf{P}(A_6(2^{k+2},2^{K-2}))=:\pi_6(2^{k+2},2^{K-2}) &\ge C(\pi_{3, HP}(1/\epsilon))^{2}\\
&\ge C\epsilon^{4}.
\end{align*}
The symbol $\pi_{3,HP}(n)$ denotes the probability that  origin is connected to distance $n$ by two disjoint open paths, and a dual vertex adjacent to the origin is connected to distance $n$ by a closed dual path, with all these paths lying in the upper half-plane $\mathbb{R}\times [0,\infty)$. In the second inequality, we have used quasimultiplicativity \cite[Proposition 12.2]{nolin}, and the universal value of the half-plane exponent \cite[Theorem 24.2]{nolin}.

\end{proof}
\end{prop}

\subsection{Probability of existence of a shortcut}\label{sec: neweventprob2}
We now prove Proposition \ref{prop: probest2}. To simplify notation, we define
\[N:= n^{\alpha_3/2},\]
where $\alpha_3$ is the constant in Proposition \ref{prop: boss}.
\begin{proof}[Proof of Proposition \ref{prop: probest2}]
Define 
\[m_N := \min\{l: (\log n)^l\ge n^{\alpha_3/8}\}\]
and
\[
M_N := \max\{l: (\log n)^l \le n^{\alpha_3/4}\}.
\]

By Lemma \ref{lma: ekprime}, it is sufficient to show that
\begin{equation}\label{eqn: bound}
\mathbf{P}(\cap_{l=m_N}^{M_N} (E_{l\lfloor \log \log n \rfloor }'(e,\epsilon))^c\mid e\in \hat{l}_n)= o\left( \frac{1}{\log n} \right)
\end{equation}
uniformly in $e \in \Lambda_n$. Note that for $\epsilon \ge \frac{1}{(\log n)^{c_2}}$ with $0<c_2<1/4$, the $E'_{l \lfloor \log \log n \rfloor}$'s are independent for different $l$, and $m_N \ge C\frac{\log n}{\log \log n}$. Thus we must deal with the effect of conditioning on $\{e\in \hat{l}_n\}$. A similar problem, with a different definition of $E_k$, was resolved in \cite{DHS15}, where it was shown \cite[Proposition 19]{DHS15} that the left-hand side of \eqref{eqn: bound} can be bounded up to a constant factor by
\[\mathbf{P}(\cap_{l=m_N}^{M_N} (E_{l \lfloor \log  \log n \rfloor }'(e,\epsilon))^c\mid A_3(e,N)),\]
where $A_3(e,N)$ is the probability that the edge $e$ is connected to $\partial B_N(e)$ by two disjoint open paths, and the dual edge $e^*$ is connected to $\partial B_N(e)^*$ by a closed dual path.
By translation invariance, it will thus be enough to show
\begin{equation}\label{eqn: whattoprove}
\mathbf{P}(\cap_{l=m_N}^{M_N} (E_{l \lfloor \log \log n \rfloor}'(\epsilon))^c\mid A_3(N)) = o\left( \frac{1}{\log n}\right).
\end{equation}

We will need the following modification of \cite[Claim 3]{DHS15}:
\begin{lma}\label{lma: circlemma}

For a sequence of integers
\[i_1<i_2< \ldots < i_k < \ldots \]
let $\tilde{B}_k$ be the event that there is a closed dual circuit with two defects (two open edges) around the origin in $Ann(i_k,i_{k+1})$ and
\[k_N=\max \{ k: i_{k+1}< n^{\alpha_3/4}\}.\]
 Furthermore, let $\hat{B}_k$ be the event that there exists an open circuit with one defect around the origin in $Ann(i_k,i_{k+1})$ and put
 \[
 C_N:= \hat{B}_1^c \cup \left( \cup_{k=2}^{k_N} \tilde{B}_k^c \right).
 \] 
 There is a choice of $i_1, i_2, \ldots$ such that $i_{k+1} \ge (\log n)^{1/4}\cdot i_k$
\[\mathbf{P}(C_N\mid A_3(N))= o\left(\frac{1}{\log n}\right).\]

\begin{proof}
We define $i_k$ by  
\begin{equation}
\label{eqn: eyek}
i_k = \left\lceil (\log n)^{c'' k}\right\rceil,
\end{equation}
where $c'' > 1/4$ is chosen so that 
\begin{align*}
&\mathbf{P}(\text{there is an open or closed dual crossing of } Ann(i_k,i_{k+1}))\\
\le &\,2\pi_1((\log n)^{c''})\\
\le &\, \frac{1}{(\log n)^2}.
\end{align*}
Recall from \eqref{eqn: pi1} that $\pi_1(m)$ is the probability that the origin is connected to distance $m$ by an open path.

With the choice \eqref{eqn: eyek}, we have
\[k_N \asymp \frac{\log n}{\log \log n}.\]
It follows that
\begin{equation*}
\mathbf{P}(A_3(N), \hat{B}_1^c\cup (\cup_{k=2}^{k_N}\tilde{B}_k^c))\le \mathbf{P}(A_3(N),\hat{B}_1^c)+\sum_{k=2}^{k_N}\mathbf{P}(A_3(N),\tilde{B}_k^c).
\end{equation*}
Using independence, the last sum is bounded above by
\[
\sum_{k=2}^{k_N} \mathbb{P}(A_3(i_k), A_3(i_{k+1},N)) \mathbb{P}(A_3(i_k,i_{k+1}),\tilde{B}_k^c).
\]
By Menger's theorem, if $A_3(i_k,i_{k+1}) \cap \tilde{B}_k^c$ occurs, then there must be three disjoint open paths and one dual closed path crossing $Ann(i_k,i_{k+1})$. Using Reimer's inequality \cite{reimer} and quasimultiplicativity of three-arm probabilities, we obtain the upper bound
\begin{align*}
C\mathbf{P}(A_3(N))\sum_{k=1}^{k_N}\pi_1(i_k,i_{k+1}) &\le Ck_N (\log n)^{-2} \mathbf{P}(A_3(N))\\
&= o(1/\log n) \cdot \mathbf{P}(A_3(N)).
\end{align*}
\end{proof}

\end{lma}

Let $0<\delta<1/4$ be equal to $\frac{1}{4}-c_2$, $F_1, F_2,\ldots $ be a maximal subcollection of the events $E'_{l\lfloor (\frac{1}{4}-\delta)\log \log n\rfloor}$, $l=1, \ldots, m_N$,
such that $F_1$ depends on the state of edges in $B(i_3)$, $F_2$ depends on the state of edges in $B(i_8)\setminus B(i_7)$, $F_3$ depends on $B(i_{13})\setminus B(i_{12})$, and so on. From \cite{DHS15}, we have the following
\begin{lma}[\cite{DHS15}, Proposition 22]
Let $r_n$ denote the minimal $k$ such that $F_k$ depends on $B(n^{\alpha_3/8})$ and $R_N$ denote the maximal $k$ such that $F_k$ depends on $B(n^{\alpha_3/4})$. Define $a_{n,N}$, by
\[a_{n,N}:= \min_{r_n \le k \le R_n} \mathbf{P}(F_k\mid A_3(N),C_N^c).\]
For some $c>0$, we have
\begin{equation}
\label{eqn: Fkest}
\mathbf{P}(\cap_{k=r_n}^{R_n}F_k^c \mid A_3(N),C_N^c) \le (1-ca_{n,N})^{R_n-r_n}.
\end{equation}
\end{lma}

Like for $k_N$,
\[R_n-r_n \ge C\frac{\log n}{\log \log n}.\]
By Proposition \ref{prop: probest} and Lemma \ref{lma: circlemma}, we have the following lower bound, independently of $k$
\begin{align*}
\mathbf{P}(E_{k}' \mid A_3(N), C_{N}^c)&\ge \mathbf{P}(E_{k}'\mid A_3(N))-\mathbf{P}(C_{N}\mid A_3(N))\\
&\ge C(\log n)^{-1+4\delta}-o(1/\log n).
\end{align*}
Thus, we obtain the lower bound
\[a_{n,N}\ge C(\log n)^{-1+4\delta},\]
and so
\begin{align*}
\mathbf{P}(\cap_{k=r_n}^{R_n}F_k^c \mid A_3(N),C_N^c) &\le \left(1-\frac{C}{(\log n)^{1-4\delta}}\right)^{C\frac{\log n}{\log \log n}}\\
&\le \exp\left( -C\frac{(\log n)^{4\delta}}{\log \log n}\right)\\
&=o(\exp(-(\log n)^{3\delta}).
\end{align*}
From this and Lemma \ref{lma: circlemma} we find
\begin{align*}
\mathbf{P}(\cap_{l=m_N}^{M_N} (E'_{(\frac{1}{4}-\delta)l\lfloor \log \log n\rfloor})^c\mid A_3(N))&\le o(\exp(-(\log n)^{3\delta})+ \mathbf{P}(C^c_{N}\mid A_3(N))\\
&= o(1/\log n).
\end{align*}
This is \eqref{eqn: whattoprove}, from which Proposition \ref{lma: circlemma} follows.
\end{proof}

\subsection{Proof of Corollary \ref{cor: quantitative}}\label{sec: lowertail}
\begin{proof}
Write, for $c_3,M>0$, 
\begin{align}
\mathbf{P}\left( S_n/L_n \ge \frac{1}{(\log n)^{c_3}} \mid H_n \right) &\le \mathbf{P}\left( S_n \ge \frac{n^2\pi_3(n)}{M(\log n)^{c_3}}\mid H_n \right) \label{eqn: S}\\
& \quad + \mathbf{P}\left( L_n \le \frac{n^2\pi_3(n)}{M}\mid H_n \right)\label{eqn: L}
\end{align}
By Theorem \ref{thm: quantitative} and Markov's inequality, the term on the right in \eqref{eqn: S} tends to zero for any $c_3<1/4$. By \cite[Lemma 4]{DHS15}, we have
\[\lim_{\epsilon \downarrow 0}\limsup_n \mathbf{P}(0<L_n < \epsilon n^2\pi_3(n))=0,\]
so \eqref{eqn: L} can be made arbitrarily small by choosing $M$ large.
\end{proof}



\end{document}